\documentclass[11pt]{amsart}

\usepackage{amsmath}
\usepackage{amssymb}
\usepackage{amsthm}
\usepackage{graphicx}
\usepackage{tikz}
\usepackage{xcolor}
\usepackage{bm}
\usepackage{hyperref}
\usepackage{comment}
\usepackage{caption}
\usepackage{subcaption}

\usepackage[margin=1.2in]{geometry}
\def\E{\mathbb{E}}

\def\Var{\mathrm{Var}}
\def\Cov{\mathrm{Cov}}

\def\E{\mathbb{E}}
\def\R{\mathbb{R}}
\def\P{\mathbb{P}}
\def\Z{\mathbb{Z}}
\def\N{\mathbb{N}}

\def\eps{\varepsilon}

\def\cP{\mathcal{P}}

\def\cF{\mathcal {F}}

\def\1{\mathbf{1}}

\def\lam {\lambda}
\def\Lam{\Lambda}
\def\tce{t_c + \eps}
\def\tce2{t_c + \frac{\eps}{2}}

\newcommand{\bbh}{\widehat{\bbeta}}
\newcommand{\blt}{\widetilde{\blambda}}
\newcommand{\bbt}{\widetilde{\bbeta}}
\newcommand{\beps}{{\bm \eps}}

\def\bX{\mathbf{X}}

\newcommand{\one}{\mathbf{1}}

\newcommand{\bt}{\mathbf{t}}
\newcommand{\bx}{\mathbf{x}}
\newcommand{\ba}{\mathbf{a}}
\newcommand{\bN}{\mathbf{N}}
\newcommand{\bB}{\mathbf{B}}
\newcommand{\bq}{\mathbf{q}}
\newcommand{\bbeta}{\bm{\beta}}
\newcommand{\balpha}{\bm{\alpha}}
\newcommand{\blambda}{\bm{\lambda}}

\newcommand{\NT}{\mathsf{NT}}

\newcommand{\QQ}{\mathcal{Q}}

\newcommand{\bsg}{\boldsymbol{\gamma}}

\newtheorem*{theorem*}{Theorem}
\newtheorem{theorem}{Theorem}
\newtheorem{lemma}[theorem]{Lemma}
\newtheorem{cor}[theorem]{Corollary}
\newtheorem{defn}[theorem]{Definition}
\newtheorem*{defn*}{Definition}

\newtheorem*{prop*}{Proposition}

\newtheorem*{conj*}{Conjecture}

\newtheorem{question}{Question}
\newtheorem*{fact*}{Fact}
\newtheorem{fact}[theorem]{Fact}
\newtheorem{assumption}{Assumption}
\newtheorem{example}{Example}

\theoremstyle{remark}
\newtheorem*{remark*}{Remark}

\newcommand{\jmin}{j_\ast}

\begin{document}
	\title{Maximum entropy and integer partitions}
	\author{Gweneth McKinley}
	\address{Department of Mathematics \\University of California San Diego}
	\email{gmckinley@ucsd.edu}
	\author{Marcus Michelen}
	\address{Department of Mathematics, Statistics, and Computer Science \\University of Illinois at Chicago}
	\email{michelen.math@gmail.com}
	\author{Will Perkins}
	\address{Department of Mathematics, Statistics, and Computer Science \\University of Illinois at Chicago}
	\email{math@willperkins.org}
	\date{\today}

	\begin{abstract}
		We derive asymptotic formulas for the number of integer partitions with given sums of $j$th powers of the parts for $j$ belonging to a finite, non-empty set $J \subset \mathbb N$.  The method we use is based on the `principle of maximum entropy' of Jaynes.  This principle leads to an  intuitive variational formula for the asymptotics of the logarithm of the number of constrained partitions as the solution to a convex optimization problem over real-valued functions. 
	\end{abstract}
	
	\maketitle

	\section{Introduction}
	\label{secIntro}
	
	An integer partition is a finite multiset $\lam$ of positive integers. It is a partition of $n$ if $\sum_{x \in \lam} x = n$. The partition number $p(n)$ counts the number of different partitions of $n$.  A classical result of Hardy and Ramanujan~\cite{hardy1918asymptotic}, obtained using Euler's generating function and the Hardy-Littlewood circle method, gives the asymptotics of $p(n)$:
	\begin{equation}
	\label{eqHR}
	p(n) = \frac{1+o(1)}{4 \sqrt{3} n }   e^{ \pi \sqrt{\frac 2 3}  \sqrt n} 
	\end{equation}
	as $n \to \infty$. 
	Since then, partitions and partition numbers have been extensively studied, and analytic, probabilistic, and combinatorial methods for analyzing partition numbers have been developed and refined.

	Several shorter or more elementary proofs of the Hardy--Ramanujan formula have since been given~\cite{erdos1942elementary,newman1962simplified}, but one can ask for an intuitive explanation of the formula; in particular, why is  the exponent  $\pi \sqrt{\frac 2 3}  \sqrt n$?
	
	We  give such an explanation here by following Jaynes' principle of maximum entropy~\cite{jaynes1957information}.  Following this principle will also allow us to determine the asymptotics of a very general class of partition numbers, those obtained by specifying  sums of various powers of the parts.  
	
 Jaynes' principle of maximum entropy is a kind of axiom about probabilistic inference: given some measurements of observed data, the best estimate for the generating distribution, in the sense of making the fewest additional assumptions,  is the distribution of maximum entropy consistent with these measurements.  More concretely, given the values of one or more statistics, the best estimate for the unknown distribution generating the data is the distribution of maximum entropy whose expectations match the observed statistics.  Jaynes explained how this principle gives an alternate derivation of the probabilistic models that arise in statistical mechanics. 
	
	To explain the application of the principle of maximum entropy to enumerating integer partitions, we will begin with the classical case $p(n)$.  Jaynes' principle suggests that to understand a typical partition of $n$, one should consider probability distributions on the  countably infinite  set of all integer partitions, and in particular, the unique probability distribution on partitions with mean sum $n$ that has the greatest entropy.   This maximum entropy distribution $\mu$ will turn out to have some remarkable properties that will help us approximate $p(n)$. 
	
	The first useful property of maximum entropy distributions is that there is an exact formula for $p(n)$ in terms of  $\mu$.  Let $\cP(n)$ denote the set of partitions of $n$.  Then
	\begin{equation}
	\label{eqpnMu}
	p(n) = e^{H(\mu)} \mu ( \cP(n)), 
	\end{equation}
	where $H(\mu) = - \sum_x \mu(x) \log \mu(x)$ is the Shannon entropy of $\mu$ and $ \mu ( \cP(n))$ is the probability that a partition drawn according to $\mu$ is a partition of $n$.  
	
	A similar formula appears in the work of Barvinok and Hartigan~\cite{barvinok2010maximum} in the context of counting integer points in polytopes (see also~\cite{barvinok2013number} and the survey~\cite{barvinok2017counting}).   Their main idea is this: to count the number of integer points in an affine subspace $A \subset \R^d$,  $| A \cap \mathbb Z^d|$, following Jaynes' principle, they construct the maximum entropy distribution $\nu$ on $\mathbb Z^d$ so that the expectation of $\nu$ lies in $A$.  They show that
	\begin{equation}
	\label{eqBarvHart}
	| A \cap \mathbb Z^d| = e^{H(\nu) } \cdot \nu(A )  \,,
	\end{equation}
 as in~\eqref{eqpnMu}.
	Thus the problem of estimating the size of the set is reduced to computing the entropy of $\nu$ and estimating $\nu(A)$, which can be done by proving a local central limit theorem under some conditions on the form of $A$.   
	
The formula~\eqref{eqpnMu} is a consequence of much more general fact about maximum entropy distributions (given in Lemmas~\ref{lemEntropyLemma} and~\ref{lemEntropyLemma2} below), itself a generalization of the fact that the (unconstrained) maximum entropy distribution on any finite set $S$ is the uniform distribution and its entropy is $\log |S|$.   
	
	The second useful property of constrained maximum entropy distributions is that they can be determined via convex programming.  This property has been used to great effect in several recent results in theoretical computer science~\cite{singh2014entropy,asadpour2017log,anari2018log} and is also used in~\cite{barvinok2010maximum}.   In the case of integer partitions, the description of $\mu$ is explicit.   We now sketch a derivation of this distribution, following a similar route to~\cite[Section 2.1]{barvinok2017counting}.  A probability distribution on partitions is a joint distribution of non-negative integer-valued random variables indexed by the natural numbers.  The constraint is that the sum of the means of these distributions times their indices equals $n$; that is, we require $\sum_{k \ge 1 } k \eta_k = n$ where $\eta_k$ is the expected number of parts of size $k$.  Since entropy is maximized by a product measure, and a geometric random variable has the greatest entropy of any non-negative integer-valued random variable with a given mean, the maximizing distribution must be a collection of independent geometric random variables.  The entropy of a geometric random variable with mean $\eta$ is 
	$$G(\eta) :=(\eta+1) \log(\eta+1) - \eta \log \eta \,,$$
	 and  so the means of these random variables are the values  $ \{ \eta_k \} _{k \ge 1}$ that maximize $ \sum_{k \ge 1} G(\eta_k)$ subject to the constraint  $\sum_{k \ge 1} k \eta_k = n$. 
	
	Distributions on partitions with independent coordinates have often arisen in the study of the structure of typical integer partitions. Indeed Fristedt identified the distribution $\mu$ above from the form of the generating function of $p(n)$~\cite{fristedt1993structure}, though he did not connect it with maximum entropy.  Vershik~\cite{vershik1996statistical,vershik1997limit} and Vershik and Yakubovich~\cite{vershik2001limit}, in the context of finding limiting shapes of partitions, considered related distributions and noted that they can be interpreted as grand canonical distributions from statistical physics.  See also~\cite{dembo1998large} in which large deviations for limit shapes are approached via the same type of distribution.   Melczer, Panova, and Pemantle~\cite{melczer2018counting} noted that while such distributions have been commonly  used  to determine limit shapes, they have only rarely been used to prove asymptotic enumeration results (their results and Tak{\'a}cs~\cite{takacs1986some} being the exceptions).  The identities~\eqref{eqpnMu} and~\eqref{eqPNmu} below provide a direct and very general link between enumeration and probability distributions on partitions.

	Given~\eqref{eqpnMu} and this representation of $\mu$ in terms of a discrete optimization problem, there are two steps to determine the asymptotics of $p(n)$: compute an accurate approximation of $H(\mu)$ and compute an accurate approximation of $\mu(\cP(n))$.

	To do  the first, we scale by $\sqrt n$ and approximate a Riemann sum by an integral  to obtain the following continuous convex optimization problem over real-valued functions:
	\begin{align}
	\label{eqHRintegral}
	M&=  \max_{f} \int_0^{\infty} G(f(x)) \, dx  \\
	\nonumber
	\text{subject to} & \quad  \int_{0}^\infty x f(x) \, dx = 1 \,,
	\end{align}
	over all integrable functions $f: [0,\infty) \to [0,\infty)$.  The optimizer $f^*(x) = \frac{1}{e^{ \frac{\pi}{\sqrt 6} x  }  -1 }$ can be found using  Lagrange multipliers.  This yields $M = \pi \sqrt{\frac 2 3}$, the constant in the exponent of the Hardy--Ramanujan formula.   To go back to the discrete problem we can take $\eta_k \approx f( k /\sqrt{n})$, and   the error in approximating the discrete optimization problem by the continuous problem can be estimated using the  Euler--Maclaurin formula, giving an additional factor $(24n)^{-1/4}$.

	For the second step, we estimate the  probability  $\mu(\cP(n))$  using a local central limit theorem, a common step in many approaches~\cite{takacs1986some,fristedt1993structure,pittel1997likely,canfield2001random,romik2005partitions,melczer2018counting}.  This gives an additional factor of $(  96 n^3)^{-1/4} $.  Multiplying $\exp \left ( \sqrt n \pi \sqrt{\frac 2 3} \right)$, $(24n)^{-1/4}$, and $(  96 n^3)^{-1/4} $  yields~\eqref{eqHR}, the Hardy--Ramanujan formula.  
	
The calculations outlined above in determining the asymptotics of $p(n)$ are not new: the extraction of the constant $\pi \sqrt{\frac{2}{3}}$ via Lagrange multipliers follows a similar path to analyzing the partition generating function using the saddle point method; the use of both the Euler--Maclaurin formula and a local central limit theorem appear in several works.  The main conceptual contribution of our perspective on the classical problem is to show that these calculations arise naturally and intuitively in the maximum entropy framework.  The exact formula~\eqref{eqpnMu} and the variational formula~\eqref{eqHRintegral} for the exponential growth rate are the two main tools that result from this perspective.

	To illustrate the utility of the maximum entropy approach we  prove asymptotic formulas for the number of partitions of a very general type: those that prescribe the sum of the $j$th powers of the parts of the partition for $j$ belonging to some finite set of non-negative integers $J$; the classical Hardy-Ramanujan case is $J = \{1 \}$, although in the same paper Hardy and Ramanujan \cite{hardy1918asymptotic} stated asymptotics for partitions of $n$ into $k$th powers for each fixed $k$ (i.e.\ the case of $J = \{k\}$) which were proven rigorously by Wright in 1934 \cite{wright1934asymptotic}. 
	
	 As above, we  give  an exact formula for the number of such partitions in terms of a constrained maximum entropy distribution and a formula for the exponential growth rate as the solution to a continuous convex optimization program.  However, when multiple sums of powers are constrained, several new wrinkles to the problem arise.  These include potential infeasibility of the constraints (related to the Stieltjes moment problem) and non-existence of a maximum entropy distribution, a well studied problem in optimization and information theory. See Section~\ref{secAssumption} below for a discussion of these issues and some new questions in the theory of integer partitions that they raise.

	Our proofs of the asymptotic formulas for moment-constrained integer partitions have three main steps, corresponding to the three factors that yield the formula~\eqref{eqHR} above in the classical case.  The first is a rigorous justification of the principle of maximum entropy to counting problems which  yields an exact formula for a partition number in terms of the maximum entropy distribution on partitions satisfying a collection of expectation constraints.  The next step is computing the asymptotics of the exponential of the entropy of this distribution by solving a continuous convex optimization problem and bounding the approximation error of a Riemann sum by an integral. The final step is  approximating the  probability that the maximum entropy distribution yields a partition satisfying all constraints by proving a multivariate local central limit theorem.  The generality of the types of partitions we enumerate necessitates some new technical ideas here.  As with many local central limit theorems, we write a probability as an integral.  The elimination of so-called \emph{minor arcs}---i.e.\ the portion of the integral that contributes an essentially negligible amount---requires a quantitative equidistribution result of Green and Tao \cite{green2012quantitative}.

	\subsection{Main results}
	
	We now describe the class of partitions we will enumerate.  Let $J$ be a finite set of non-negative integers containing at least one positive integer, and let $\mathbf N = (\bN_j)_{j \in J}$ be a vector of positive integers indexed by $J$.  A partition $\lam$ has \emph{profile} $\mathbf N$ if 
	$$\sum_{x \in \lam} x^j = \bN_j \qquad \text{ for all } j \in J.$$
	We call $J$ the \emph{profile set}.

	Let $\cP(\bN)$ denote the set of partitions with profile $\bN$ and let $p(\mathbf{N}) = | \cP (\mathbf N)|$.  For instance, with $J = \{1\}$ and $\bN_1 =n$, we have $p(\mathbf{N})=p(n)$, the usual partition number.  To study the asymptotics of $p(\mathbf{N})$ we normalize the profile.  For $\balpha \in \R_+^J$ and $n \in \N$, let 
	\begin{align*}
	\bN(\balpha, n) &= ( \lfloor \balpha_j n^{(j+1)/2} \rfloor)_{j\in J} \,.
\end{align*}
Then let  $\tilde p_n(\balpha) = p (\bN(\balpha, n))$.   We will study the asymptotics of $\tilde p_n(\balpha)$ as $n \to \infty$.   The scaling is chosen to obtain a non-trivial limit shape, which we discuss further in Section~\ref{secLimitShapes}.  

This general class of partition problems includes several specific cases studied previously:
\begin{itemize}
\item  The classical case, partitions of $n$, is obtained by taking $J= \{1\}$, $\balpha_1=1$.
\item Partitions of an integer into sums of $k$th powers~\cite{wright1934asymptotic} is obtained by taking $J= \{k\}$, $\balpha_k=1$.
\item Partitions of $n$ with a given number of parts~\cite{szekeres1953some,canfield1997recursions,romik2005partitions} is obtained by taking $J = \{0,1\}$, $\balpha_0=c$, $\balpha_1=1$. 
\end{itemize}
While these cases are all covered by our main result, several new features of the problem emerge once the set  $J$ includes more than one positive integer, and to the best of our knowledge such cases have not been considered before.  One new feature is that  certain profiles are impossible, either for number-theoretic reasons or because the values in $\balpha$ are incompatible.  The number-theoretic constraints pose some new challenges in proving the local central limit theorem (Section~\ref{secCLT}).  Additionally, in this case the continuous convex program analogous to~\eqref{eqHRintegral} may not have a solution even when it is feasible -- this is closely related to the problem of the existence of maximum entropy distributions, which has a long history in the study of infinite-dimensional convex optimization.  We discuss each of these features in what follows, starting with the constraints on profiles.  

As just mentioned, some profiles are impossible due to the incompatibility of the constraints; for example, the constraints may violate the Cauchy-Schwartz inequality.  The compatibility of constraints depends on the vector $\balpha$ and can be expressed in terms of the Stieltjes moment problem~\cite{stieltjes1894recherches}; this is discussed further in Section~\ref{secAssumption}.
	
	Other profiles are impossible for number-theoretic reasons that depend on $n$.  For instance, since $k^2 \equiv k \mod 2$ for all integers $k$, we have that $p(\bN) = 0$ if $\bN_1 \not\equiv \bN_2 \mod 2$.  A concise way of describing this particular constraint is that the polynomial $\frac{1}{2}x^2 + \frac{1}{2}x$ is \emph{integer-valued} meaning that $ \frac{1}{2}m^2 + \frac{1}{2}m \in \Z$ for all $m \in \Z$.  Thus a necessary condition for  $\cP(\mathbf{N}) \neq \emptyset$ is that $\frac{1}{2}\bN_2 + \frac{1}{2} \bN_1 \in \Z$.   It will turn out that \emph{all} number-theoretic obstructions can be defined in this way.  
	Let
	\begin{equation} 
	\label{eqQJ}
	\QQ_J = \left\{ \sum_{j \in J} t_j x^j:\ t_j \in (-1/2,1/2] \text{ and } \sum_{j \in J} t_j m^j \in \Z \text { for all }m \in \Z  \right\}
	\end{equation}
 be the set of integer-valued polynomials using only powers in $J$ and having coefficients in $(-1/2,1/2]$.  We subsequently define the set $$\NT := \NT(J) = \left\{ (m_j)_{j \in J} \in \Z^{J} : a_j m_j \in \Z \text{ for all } \sum_{j \in J} a_j x^j \in \QQ_J \right\}\,.$$
	It follows from the definition that if $\mathbf N \notin \NT$ then $p(\mathbf{N}) = 0$. We say $\balpha$ is $n$-feasible if $\bN( \balpha,n) \in \NT$.

	 To apply the principle of maximum entropy to $p (\mathbf N)$, we define $\mu$ to be the maximum entropy distribution on the set of all partitions so that 
	$$ \E_{\lam \sim \mu} \sum_{x \in \lam} x^j = \bN_j $$
	for all $j \in J$.   As in the special case above, we will see that the maximum entropy distribution can be represented as a collection of independent geometric random variables.  Under one assumption on $\balpha$, we will show, as in~\eqref{eqpnMu},  the exact formula
	\begin{equation}
	\label{eqPNmu}
	p(\mathbf N) = e^{H(\mu)} \mu( \cP(\mathbf N)) \,.
	\end{equation}	
	The assumption we make on $\balpha$ ensures feasibility of the moment constraints and facilitates the maximum entropy method.
	\begin{assumption}
	\label{Assumption1}
	There exists $\bbeta \in \{\R \setminus 0 \}^J$ so that
	\begin{equation}\label{eq:geo-cont}
	\int_0^\infty \frac{x^j}{\exp\left( \sum_{\ell \in J} \bbeta_\ell x^\ell \right) - 1}\,dx = \balpha_j \quad\text{for }j \in J \,.
	\end{equation}
	\end{assumption}
	In fact, for these integrals to converge, $ \bbeta$ must belong to a certain convex subset of $\R^J$: those vectors for which the polynomial $ \sum_{j \in J} \bbeta_j x^j$ is positive on $(0,\infty)$.  
We discuss Assumption~\ref{Assumption1} and its connections to other problems in optimization and information theory in Section~\ref{secAssumption}.

	Equipped with~\eqref{eqPNmu} and Assumption~\ref{Assumption1}, we  pose a continuous convex program that determines the exponential growth rate of $ \tilde p_n(\balpha)$ (in $\sqrt n$).  Recall that $G(\eta) = (\eta+1) \log(\eta+1) - \eta \log \eta$.  We define    
	\begin{align}
	\label{eqMbsb}
	&M(\balpha) = \max_{f \in \cF}  \int_{0}^\infty G(f(x)) \, dx \\
	\nonumber
	\text{subject to }  
	\quad &\int_0^\infty x^j f(x) \, dx = \balpha_j  \text{ for } j \in J \,,
	\end{align}
	where $\cF$ is the set of all integrable functions $f: [0,\infty) \to [0,\infty)$.  Note that the objective function to be maximized is strictly concave, and the constraints linear, so we have a convex program.   Assumption~\ref{Assumption1} implies this optimization problem is feasible; in fact via convex duality the vector $\bbeta$ is a certificate that  the optimizer is 
\begin{equation}
\label{eqfstar}
f^*(x) :=\frac{1}{   \exp\left( \sum_{j \in J} \bbeta_j x^j \right) - 1 } \,.
\end{equation}
 The optimum $M(\balpha)$ determines the growth rate of the  entropy of the maximum entropy distribution $\mu$ from~\eqref{eqPNmu}, and thus the growth rate of $\tilde p_n(\balpha)$.  We can now state our main result.

	\begin{theorem}
		\label{thmIntegralNonDistinct}
		For all profile sets $J$ and all $\balpha \in \R_+^J$ satisfying Assumption~\ref{Assumption1},
		\[  \tilde p_n(\balpha) = (c(\balpha)+o(1)) \cdot \frac{ e^{\sqrt n M(\balpha)} }{  n^{b(J)} } \]
		when $\balpha$ is $n$-feasible (and $0$ otherwise). 
	\end{theorem} 
	
	The constant $b(J)$ is given by
	\begin{align*}
	b(J) &= \frac{\jmin + |J|}{4} + \frac{1}{2}\sum_{j \in J} j   \, ,
	\end{align*}
	where $\jmin := \min  J$.  The constant $c(\balpha)$ depends on  $\balpha$ implicitly through the vector $\bbeta$  guaranteed by Assumption~\ref{Assumption1}  and is given by 
		\begin{align*} 
		c(\balpha) &= \frac{   |\QQ_J|   }{ (2\pi)^{ \frac{\jmin + |J|}{2}} ( \det {\Sigma}) ^{1/2} }  \cdot {\bbeta}_{\jmin}^{   \frac{\one_{\jmin \geq 1}}{2}   }   \cdot \exp \left ( \frac{\one_{\jmin = 0}}{2}\left(\frac{{\bbeta}_0}{e^{{\bbeta}_0} - 1} - G \left (\frac{1} { e^{{\bbeta}_0}-1} \right  ) \right)  \right)
		\end{align*}
	 where
	\begin{equation}
	\label{eqSigmaDef}
	{\Sigma} = \left(\int_0^\infty x^{i + j} \frac{\exp\left( \sum_{\ell \in J} \bbeta_\ell x^\ell \right)}{\left(\exp\left(\sum_{\ell \in J} \bbeta_\ell x^\ell \right) - 1\right)^2} \,dx\right)_{i,j \in J} 
	\end{equation}
	and $\QQ_J$ is defined in~\eqref{eqQJ}.

\subsection{Limit Shapes}
\label{secLimitShapes}
The scaling $\bN(\balpha,n) = (\lfloor \balpha_j n^{(j+1)/2} \rfloor)_{j \in J}$ is chosen so that a typical partition $\lambda$ in $\mathcal{P}(\bN)$ has a limit shape.  In particular, if we rescale the Young diagram of $\lambda \in \mathcal{P}(\bN)$ by $\sqrt{n}$ in each direction, then the area of the rescaled diagram will be of roughly constant order; indeed, in the case that $1 \in J$, the rescaled area will be exactly $\balpha_1$.  Informally, we say that there is a \emph{limit shape} if  the rescaled Young diagram of a uniformly random $\lam \in \mathcal{P}(\bN)$ converges in distribution (in an appropriate sense) to a constant shape.  In the classical case of $J = \{1\}$, a limit shape was shown to exist by Szalay and Tur\'an \cite{st-1,st-2} (see also \cite{vershik1996statistical}).  This shape is shown in Figure~\ref{figHRshape}.  Similarly, a limit shape for partitions whose Young diagram fit in a rectangle of constant aspect ratio was found by Petrov \cite{petrov}.  

\begin{figure}
\centering
\begin{minipage}{.5\textwidth}
  \centering
  \includegraphics[width=1\linewidth]{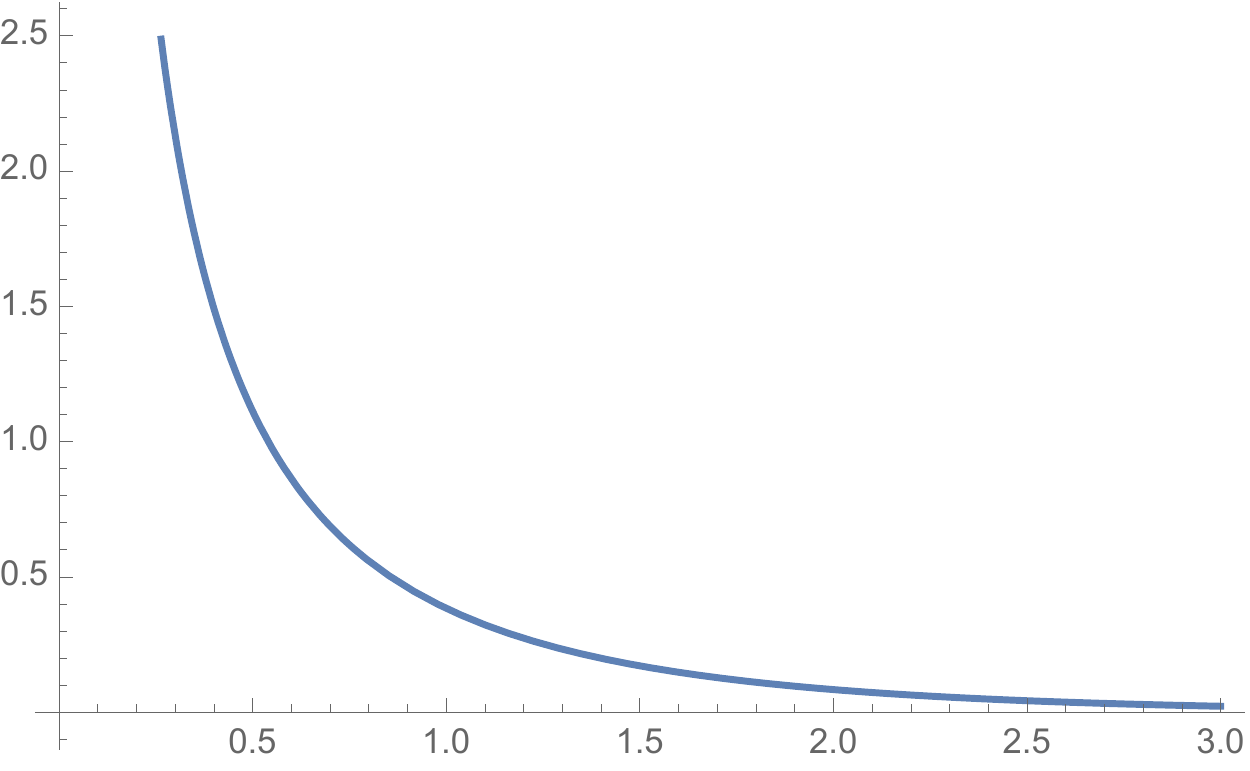}
  \captionof{figure}{The optimizer $f^* =  (e^{ \frac{\pi}{\sqrt 6} x  }  -1 )^{-1}$ for partitions of $n$.}
  \label{fig:fHR}
\end{minipage}%
\begin{minipage}{.5\textwidth}
  \centering
  \includegraphics[width=1\linewidth]{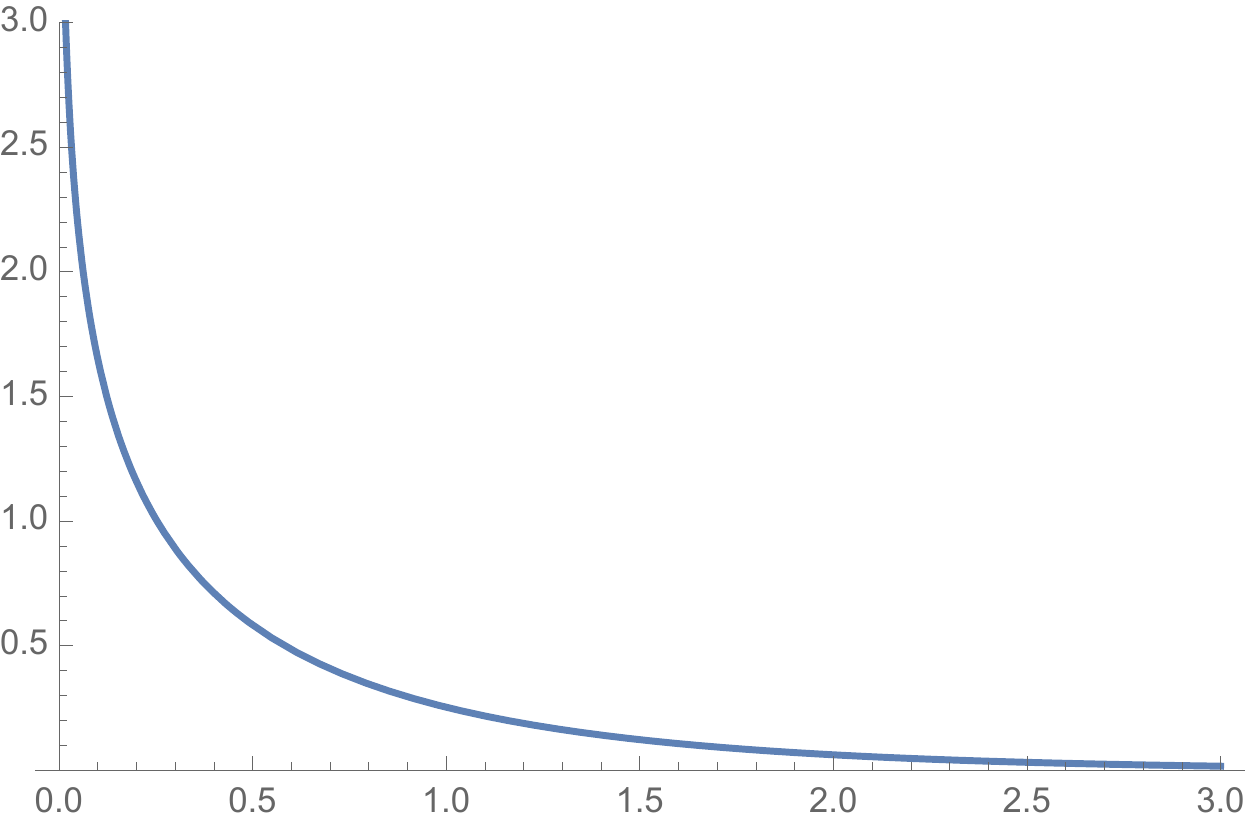}
  \captionof{figure}{The limit shape for partitions of $n$.}
  \label{figHRshape}
\end{minipage}
\end{figure}

We will show that there is a limit shape for all the cases covered by Theorem~\ref{thmIntegralNonDistinct}.  In order to state precisely what is meant by ``limit shape'' some preliminaries are required. Following Vershik \cite{vershik1996statistical}, define the space $\mathcal{D} = \{  \phi(t)) \}$ where $\phi: \R_+ \to [0,\infty) $ with $\int \phi(t)\,dt <\infty$ and $\phi$  non-increasing.  Endow $\mathcal{D}$ with the topology of uniform convergence on compact sets.  We will think of $\mathcal{D}$ as the space of scaled Young diagrams and their limits, where we consider Young diagrams in French notation.
For a partition $\lambda$ and $n \in \N$, define the function \begin{equation*}
\phi_{\lambda,n}(t) = n^{-1/2} \left| \{ a \in \lambda: a \geq t \sqrt{n} \}\right|\,.
\end{equation*}
The function $\phi_{\lambda,n}(t)$ is simply the boundary of the Young diagram of $\lambda$ in French notation, rescaled by $\sqrt{n}$ in each direction.  Our goal is to identify the limit shape when $\lambda$ is chosen from $\mathcal{P}(\bN)$ uniformly at randomly; intuitively, the law of large numbers states that if $\lambda$ is chosen from the maximum entropy measure $\mu_n$ instead, then we have 
\begin{equation}\label{eq:LLN}
n^{-1/2} \left| \{ a \in \lambda: a \geq x \sqrt{n} \}\right| \approx  \int_x^\infty  \frac{1}{\exp(\sum_{j \in J}\bbeta_j s^j ) - 1 }\,ds \,.
\end{equation}

With this in mind, the function $\phi_\infty(t)$ defined via $$\phi_\infty(t) =  \int_{t}^\infty \frac{1}{\exp(\sum_{j \in J}\bbeta_j s^j ) - 1 }\,ds = \int_t^\infty f^*(s)\,ds ,$$
where $f^*$ is as in~\eqref{eqfstar},
is a strong candidate for the limit shape.  This will turn out to be the case.

\begin{theorem}\label{thm:shape}
	In the context of Theorem~\ref{thmIntegralNonDistinct}, let $\lambda$ be an element of $\mathcal{P}(\bN(\balpha,n))$ chosen uniformly at random.  Then $\phi_{\lambda,n}$ converges in distribution to $\phi_\infty$ as $n \to \infty$.
\end{theorem}
	
Note that the value $M(\balpha)$ may be viewed as a functional of the limit shape $\phi_\infty$ itself, since $f^*$ can be obtained by differentiation.  The relationship between the growth rate of $\exp(M(\balpha)\sqrt{n})$ and the limit shape $\phi_\infty$ is not new and has a long history in statistical mechanics and its adjacent fields.  For instance, in a survey on the limit shapes, Shlosman shows that the asymptotic $\log p(n) \sim \pi \sqrt{2/3} \sqrt{n}$ follows from the shape theorem for partitions~\cite{shlosman}.  The survey \cite{okounkov} by Okounkov discusses many other examples of relationships---both heuristic and rigorous---between limit shapes, asymptotic enumeration and large deviation principles.

We now give two examples of limit shapes obtainable in Theorem~\ref{thm:shape}. These examples were obtained by choosing $\bbeta$ first then calculating the corresponding $\balpha$.  
\begin{example}
\label{exShape3}
Let $J= \{0,1,2,3, 4\}$ and let $\balpha =  \{12.8748, 6.698, 4.66192, 3.72617, 3.15877 \}$.  Then $\bbeta \approx \{.95, -10.1, 36.5, -49.5, 22.4   \}$.  The limit shape is given in Figure~\ref{fig:shape3}. 
\end{example}

\begin{figure}
\centering
\begin{minipage}{.5\textwidth}
  \centering
  \includegraphics[width=1\linewidth]{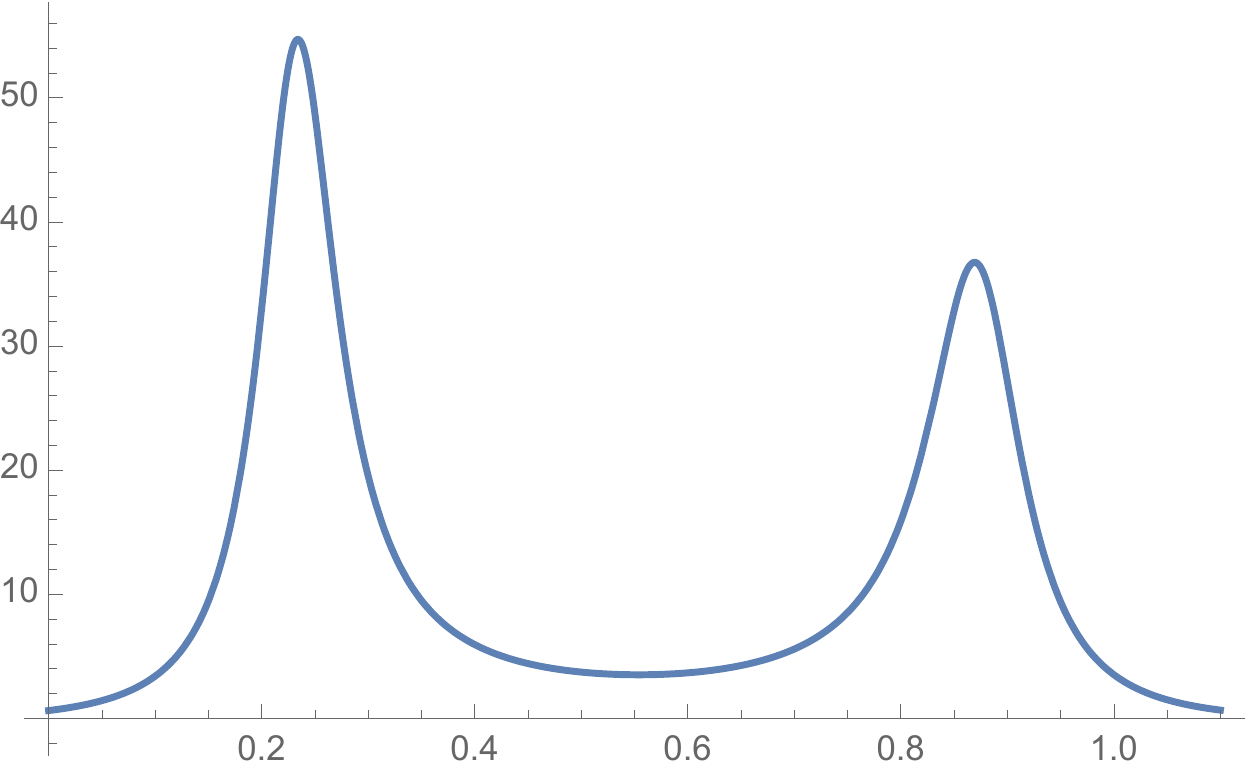}
  \captionof{figure}{The optimizer $f^*$ for Example~\ref{exShape3}.}
  \label{fig:f3}
\end{minipage}%
\begin{minipage}{.5\textwidth}
  \centering
  \includegraphics[width=1\linewidth]{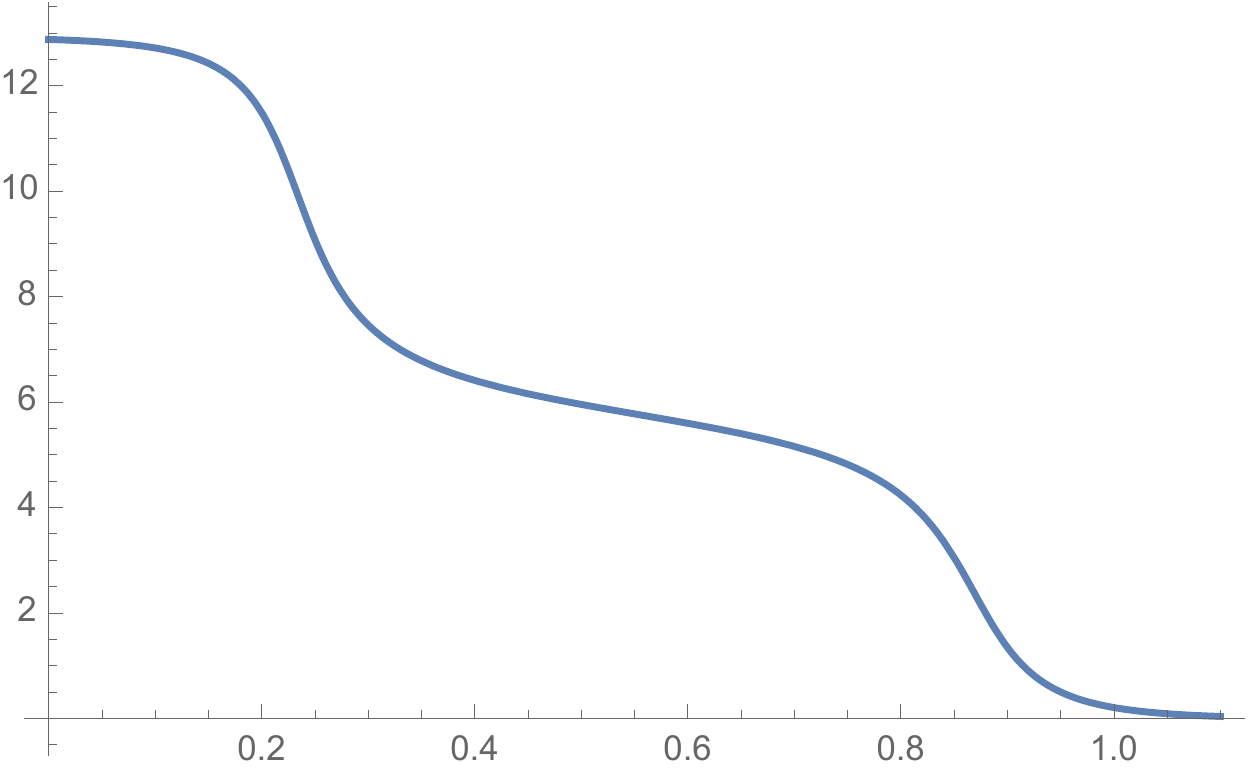}
  \captionof{figure}{The limit shape for Example~\ref{exShape3}.}
  \label{fig:shape3}
\end{minipage}
\end{figure}

\begin{example}
\label{exShape4}
Let $J= \{1,2,3\}$ and let $\balpha =  \{4.31168, 3.86652, 3.65774 \}$.  Then $\bbeta \approx \{4.0, -8.5, 4.6 \}$.  The optimizer $f^*$ is shown in Figure~\ref{fig:f4} and the limit shape is show in Figure~\ref{fig:shape4}.   Note that both have a vertical asymptote at $0$ indicating that the typical number of parts of such a partition is $ \omega(\sqrt{n})$. 
\end{example}

\begin{figure}
\centering
\begin{minipage}{.5\textwidth}
  \centering
  \includegraphics[width=1\linewidth]{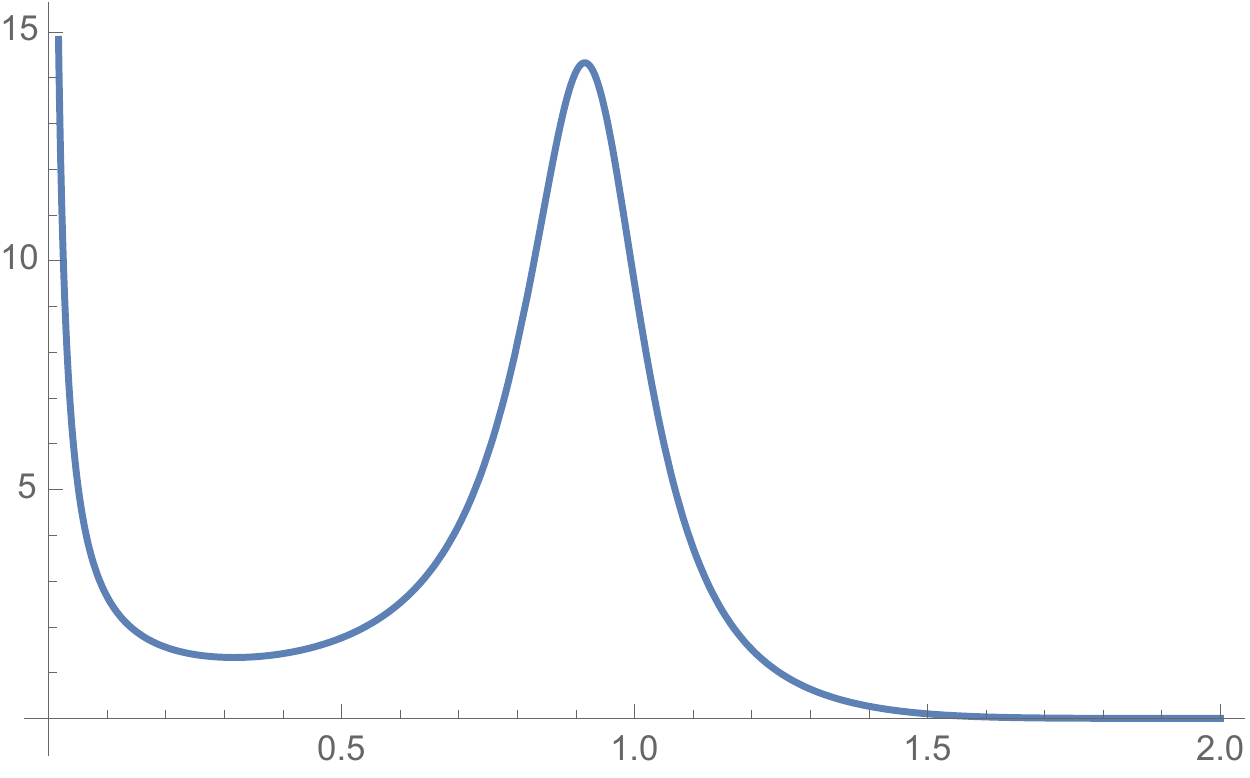}
  \captionof{figure}{The optimizer $f^*$ for Example~\ref{exShape4}.}
  \label{fig:f4}
\end{minipage}%
\begin{minipage}{.5\textwidth}
  \centering
  \includegraphics[width=1\linewidth]{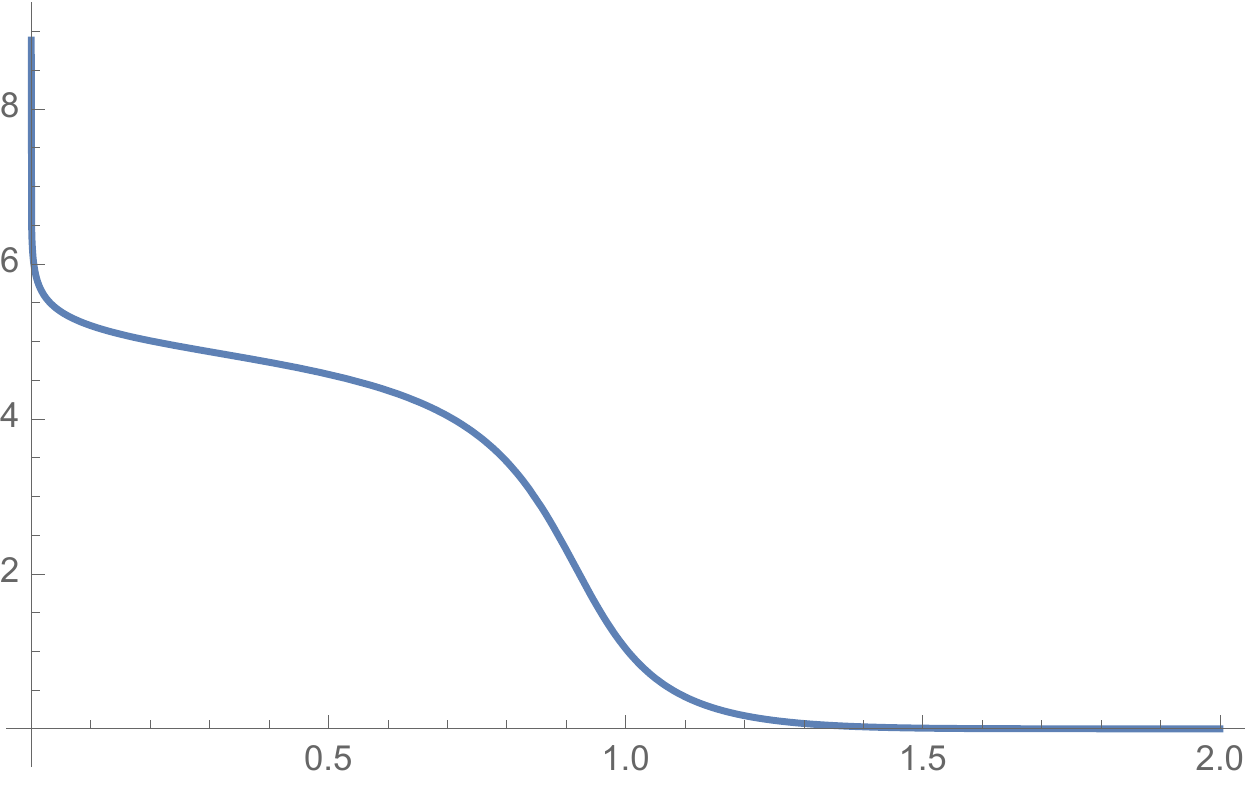}
  \captionof{figure}{The limit shape for Example~\ref{exShape4}.}
  \label{fig:shape4}
\end{minipage}
\end{figure}

These examples indicate some of the  rich behavior possible in the setting of moment-constrained integer partitions. By specifying $k$ moments we can obtain a limit shape with up to $k-1$ inflection points.  In fact, it is not hard to show that the set of limit shapes obtainable in the framework of Theorem~\ref{thmIntegralNonDistinct} is dense in the set of all integrable, non-negative and non-increasing functions on $[0, \infty)$.  
	\subsection{Remarks on Assumption~\ref{Assumption1}}
	\label{secAssumption}
	
	The Stieltjes moment problem~\cite{stieltjes1894recherches} is the problem of finding a density function of a continuous random variable supported on $[0,\infty)$ given its  
	moments.  In other words, given a sequence of positive numbers $\balpha_1, \balpha_2, \dots $, determine  if there is a density  $f$ on $[0,\infty)$ with $\int x^k f(x) \, dx = \balpha_k$, $k \ge1$, and if so, whether it is uniquely determined.  The truncated or reduced Stieltjes moment problem is the same but with only the first $d$ moments specified.   The closely related Hausdorff and Hamburger moment problems are the analogous problems with support $[0,1]$ and $(-\infty, \infty)$ respectively.  Stieltjes determined necessary and sufficient conditions on the sequence $\balpha_1, \balpha_2, \dots$ for the existence of a solution to the moment problem.  Let $A_n$ and $\overline A_n$ be the Hankel matrices generated from $1 = \balpha_0, \balpha_1, \dots, \balpha_n $ and $\balpha_1, \balpha_2, \dots, \balpha_n $ respectively.  Then the Stieltjes moment problem has a solution if and only if $ \det A_n > 0$ and $\det \overline A_n > 0$ for all $n$.   The conditions for the truncated Stieltjes problem are the same but only constraining the Hankel matrices formed using the specified moments.
	
	When the (truncated) moment problem is feasible there may be infinitely many solutions, and so one can ask for a principled approach to select one distribution satisfying the given constraints.  Jaynes' principle of maximum entropy suggests choosing the distribution with maximum entropy subject to the constraints.   Such a choice is very natural: many widely used distributions, both discrete and continuous, are maximum entropy distributions subject to constraints on the support and a small number of moments, e.g.\ Gaussian, exponential, geometric, and uniform (see e.g.~\cite[Chapter 12]{cover1999elements}). 

One can pose the maximum entropy Stieltjes moment problem as a continuous convex program.  	
\begin{align}
	\label{maxEntMoments}
	&\max_{f\in\cF}   -\int_{0}^\infty f(x) \log f(x) \, dx \\
	\nonumber
	&\text{subject to }  
	\quad \int_0^\infty x^j f(x) \, dx = \balpha_j \text{ for } j =0, \dots d \,,
	\end{align}
	  where we set $\balpha_0 =1$ to ensure that $f$ is a probability density function.  The objective function is the entropy of the distribution with density $f$.   Note the similarity of the maximum entropy moment program~\eqref{maxEntMoments} to the program~\eqref{eqMbsb}, which we will call the \textit{maximum geometric entropy moment problem}.   The only difference is in the objective functions which are different strictly concave functions (also in~\eqref{eqMbsb} we may only specify a subset of the first $m$ moments, but we could do the same in~\eqref{maxEntMoments}).  The two problems share essentially all of their qualitative features.    To describe these features, let us assume for now that $J = \{ 0, \dots , d\}$ and that $\balpha_0 =1$ (the latter is simply a normalization).  Then the feasible sets of~\eqref{eqMbsb} and~\eqref{maxEntMoments} are identical and non-empty if and only if the Stieltjes condition holds.   Csisz{\'a}r~\cite{csiszar1975divergence} shows that if a maximum entropy solution to~\eqref{maxEntMoments} exists then it must be of the form $f(x) = \exp( - \sum_{j=0}^d \bbeta_j x^j )$ for some $\bbeta \in \R^{d+1}$.   Moreover, from any solution $\bbeta$ to the system of equations 
\begin{equation}
\label{eqEntropySystem1}
\int_0^\infty x^j    \exp \left(- \sum_{\ell=0}^d \bbeta_\ell x^\ell   \right)  \, dx = \balpha_j \text{ for } j =0, \dots d 
\end{equation}
we can generate an optimal solution to~\eqref{maxEntMoments} via $f(x) =  \exp \left(- \sum_{j=0}^d \bbeta_j x^j   \right) $.   We will show below in Lemma~\ref{lemContinuousSolution} that from any solution $\bbeta$ to the system~\eqref{eq:geo-cont} we can also generate an optimal solution to~\eqref{eqMbsb}.

On the other hand, it can happen that $\balpha$ satisfies the Stieltjes condition for feasibility but no solution to~\eqref{eqEntropySystem1} exists~\cite{junk2000maximum,tagliani2003maximum}.  A simple example is $\balpha = (\balpha_0, \balpha_1, \balpha_2)= (1,1,3)$.   This set of moments is feasible for the truncated Stieltjes problem since the matrices $\{ \{1,1 \},\{ 1,3\} \}$ and $\{ \{1\}\}$ have positive determinants.   In this case there is a least upper bound to the maximization problem but it is not achieved by any density function.
	
Similarly, in the case of the maximum geometric entropy moment problem there are vectors $\balpha$ feasible for the Stieltjes moment problem that nevertheless do not satisfy Assumption~\ref{Assumption1} and thus do not have an optimal solution to~\eqref{eqMbsb}.  It is straightforward to generate such an $\balpha$.   First pick $\bbeta = (\bbeta_0, \bbeta_1, \dots \bbeta_d)$ so that $\sum_{j=0}^d x^j \bbeta_j$ is positive on $(0,\infty)$.   Let $\balpha'_j$, $j=0, \dots, d+1$ be given by
\begin{align*}
\balpha'_j &= \int_{0}^\infty \frac{x^j}{ \exp \left( \sum _{\ell=0}^d x^\ell \bbeta_\ell \right) -1} \,.
\end{align*}
Then let $\balpha_j = \balpha'_j$ for $j=1, \dots, d$ and $\balpha_{d+1} = \balpha_{d+1}' +\eps$ where $\eps>0$ is chosen small enough so that the corresponding Hankel matrices have positive determinant (this mirrors the construction in~\cite{junk2000maximum}).  The results in a feasible optimization problem without an optimal solution.  Such vectors $\balpha$ are not covered by our results and we ask if it is still possible to determine the asymptotics of $\tilde p_n (\balpha)$.

\begin{question}
\label{questionNoOpt}
Fix $J$ and suppose $\balpha$ is feasible for the Stieltjes moment problem but no optimal solution to~\eqref{eqMbsb} exists.  What are the asymptotics of $\tilde p_n (\balpha)$ as $n \to \infty$?
\end{question}

One can also ask a computational question: given $\balpha$ satisfying Assumption~\ref{Assumption1}, can we efficiently compute the corresponding $\bbeta$ and thus the growth constant $M(\balpha)$?   Again essentially all of the work devoted to solving the analogous maximum entropy moment problem can be applied here, since both objective functions are strictly concave.  We refer the reader to~\cite[Chapter 12]{lasserre2010moments}.

	\subsection{Extensions}
	There is a wealth of  extensions and generalizations of the problem of enumerating integer partitions (see e.g.\ \cite{andrews1998theory}).   Many of these extensions can be framed in the maximum entropy framework, leading to continuous convex  programs like~\eqref{eqMbsb} that express the exponential growth rate of the partition number.   Once we have such a  program it is natural to modify it by adding additional constraints or restricting the domain of the candidate functions.  These new optimization problems can be translated back to give new classes of integer partitions problems.   Posing the problems in the maximum entropy framework gives a natural, unifying explanation of many methods and formulas in the literature, but also illuminates some new connections between integer partitions and infinite-dimensional convex programming.  Below we indicate some possible extensions of these methods to previously studied classes of integer partitions.

	\subsubsection{Distinct partitions: changing the objective function}
	\label{secDistinctPart}

	One way to restrict a class of partitions is to insist that each part appear with one of a set of  prescribed multiplicities.  The simplest such case is that of  distinct partitions: partitions in which each part occurs with multiplicity at most $1$; equivalently a distinct partition is a subset of $\mathbb N$ (as opposed to a multiset).    Let $q(n)$ denote the number of distinct partitions of $n$.  The asymptotics of $q(n)$ have been studied in, e.g.~\cite{meinardus1954partitionen,romik2005partitions}.  
	
	Given  a profile $\mathbf N$ indexed by the profile set $J$, let $q(\bN)$ be the number of distinct partitions $\lam$ so that $\sum_{x \in \lam} x^j = \bN_j $ for all $j \in J$, and for $ \balpha \in \R_+^J $ let $\tilde q_n( \balpha) =  q(\bN(\balpha, n)) $.   Using the methods of this paper, and making an assumption analogous to Assumption~\ref{Assumption1}, we believe one could show the exponential growth rate of $\tilde q_n( \balpha) $ is again given by the optimum of a continuous convex program:
	\begin{align*}
	&M_{\mathrm{dist}}(\balpha) = \max_{f \in \cF_1}  \int_{0}^\infty H(f(x)) \, dx \\
	\text{subject to }  
	\quad &\int_0^\infty x^j f(x) \, dx = \balpha_j  \text{ for } j \in J \,,
	\end{align*}
	where $\cF_1$ is the set of all integrable functions $f: [0,\infty) \to [0,1]$ and $H(p) = -p \log p - (1-p) \log(1-p)$ is the entropy of a Bernoulli random variable with parameter $p$.  This is exactly the same optimization problem as in~\eqref{eqMbsb} but with a different objective function; both objective functions are strictly concave, however, and so share essentially all the same qualitative properties.

	\subsubsection{Bounded Young diagrams: restricting the support}
	\label{secRectangle}
	
	Another class of restricted partitions is the class of partitions with bounded Young diagrams.   That is, partitions with restrictions on the size of the largest part and on the number of parts.  Asymptotic enumeration of partitions of $n$ with largest part and number of parts both $\Theta(\sqrt{n})$ has been carried out in~\cite{jiang2019generalized,melczer2018counting}.  
	
	In particular, the method used by Melczer, Panova, and Pemantle in~\cite{melczer2018counting} shares some important  steps in common with our approach here: both are probabilistic approaches to enumeration, and both solve a limiting variational problem.  Melczer, Panova, and Pemantle enumerate partitions by proving a local large deviation principle, which, when solved, produces a probability distribution on partitions given by independent geometric random variables with specified means.   In fact computing a large deviation rate function in this setting is equivalent to minimizing the Kullback-Leibler divergence between probability measures (as in Sanov's Theorem~\cite{sanov1958probability}), which is essentially an entropy maximization problem (see the discussion in, e.g.~\cite{csiszar1975divergence}). Where the maximum entropy and large deviation approaches differ is that a large deviation approach requires a prior distribution on partitions and thus is restricted to settings such as that of bounded Young diagrams, while the principle of maximum entropy works in general, without a prior (and in fact this is an important motivation for the principle itself: to be able to generate a prior when one does not exist). 
	
	One can follow the methods of this paper to enumerate partitions with a given profile and largest part  at most $A \sqrt{n}$. The continuous convex program giving the growth rate is the following:
	\begin{align*}
	&M_{A}(\balpha) = \max_{f }  \int_{0}^A G(f(x)) \, dx \\
	\text{subject to }  
	\quad &\int_0^A x^j f(x) \, dx = \balpha_j  \text{ for } j \in J \,,
	\end{align*}
	In regards to the discussion in Section~\ref{secAssumption} about the existence of a solution, something different happens in this setting.  There are again feasibility conditions on the moments $\balpha$; this time related to the Hausdorff moment problem~\cite{hausdorff1921summationsmethoden} (that of finding a probability distribution on a bounded interval with given moments).  Because the interval $[0,A]$ is compact, if a collection of moments are feasible for the truncated Hausdorff moment problem, then a unique maximum entropy distribution with the given moments always exists~\cite{mead1984maximum}.  The same is true if we maximize geometric entropy, and so the only assumption on $\balpha$ needed is feasibility: the situation discussed in Section~\ref{secAssumption} cannot occur.  
	
	\begin{question}
		Can Theorem \ref{thmIntegralNonDistinct} be shown for partitions in a rectangle under the weakened assumption that $\balpha$ is feasible for the Stieltjes moment problem?
	\end{question}
	
	\subsubsection{Plane Partitions and Higher Dimensions}
	The Young diagram of a partition gives us a two-dimensional representation of the partition $n$; the natural generalization to dimension three --- functions $f: \N^3 \to \{0,1\}$ that are weakly decreasing in each coordinate---give the structures known as \emph{plane partitions}.  The asymptotics for the number of plane partitions of weight $n$ was given by Wright \cite{wright1931asymptotic}, although enumeration of restricted plane partitions appears to not have been explored.  
	
	\begin{question}
		What are the asymptotics for the number of plane partitions that fit in an appropriately scaled rectangular box?
	\end{question}
	
	\subsection{Organization and notation}
In Section~\ref{secEntropy} we discuss the connection between maximum entropy distributions and counting and prove the exact formula~\eqref{eqPNmu}. 	  In Section~\ref{secOPT} we  compute the asymptotics of $e^{H(\mu)}$ by comparing a discrete optimization problem to the continuous optimization problem~\eqref{eqMbsb}.  In Section~\ref{secCLT} we prove a multivariate local central limit theorem to estimate the factor $\mu ( \cP(\bN)) $.  In Section~\ref{secLimitShape} we prove Theorem \ref{thm:shape} which shows the existence of a limit shape.

	All logarithms in this paper are base $e$.  The Shannon entropy of a discrete random variable $X$ with probability mass function $f_X$ is  $H(X) = - \sum_x f_X(x) \log f_X(x)$.    A geometric random variable $X$ with parameter $p \in (0,1)$ has probability mass function $f_X(k) = p(1 - p)^k$ for $k \geq 0$.  Its mean is $\eta = \frac{1-p}{p}$ and its entropy is $(\eta+1) \log (\eta+1) - \eta \log \eta$.   We let $\cP$ denote the set of all partitions, which we  identify with the set $L_0(\N) := \bigoplus_{j = 1}^\infty \N_0$, i.e.\ the set of sequences in $\N_0:= \N\cup \{0\}$ that converge to $0$ (in particular, $\cP$ is a countable set).      We let $\R_+ = (0, \infty)$.   We will use the convention that bold symbols ($\bN, \balpha, \bbeta, \dots$) denote vectors indexed by a profile set $J$ or by the integers $\{1, \dots , d\}$.

	\section{Maximum entropy distributions}
	\label{secEntropy}
	
	In this section we derive the maximum entropy distribution on partitions given moment constraints and give an exact formula for $p(\mathbf{N})$ in terms of this distribution.  
	
	We first give an elementary and completely general statement connecting counting to maximum entropy. 
	\begin{lemma}
		\label{lemEntropyLemma}
		Let $\Omega$ be a finite set and let $f : \Omega \to \R^d$.   For $\bB \in \R^d$ define \\ $\Lam_{\bB} = \{ \omega \in \Omega: f(\omega) = \bB \}$ and suppose $\Lam_{\bB}$ is finite and non-empty.   Let $\mu$ be the maximum entropy distribution on $\Omega$ so that $\E_{\mu} f =\bB$.  Then
		\begin{equation}
		\label{eqEntropyGeneral}
		| \Lam_{\bB}|  = e^{H(\mu)}  \cdot \mu(\Lam_{\bB}) \,.
		\end{equation}
	\end{lemma}
	This is inspired by, e.g.~\cite[Theorem 3.1]{barvinok2010maximum}, and is a simple consequence of the form of maximum entropy distributions subject to mean constraints.  The use of convex duality in the proof comes from~\cite{boyd2004convex,singh2014entropy}.
	\begin{proof}
		Let $K \subseteq \R^d$ be the convex hull of the set $\{ f(\omega): \omega \in \Omega   \}$.  We have $\bB \in K$ since $\Lam_{\bB}$ is non-empty.   We may assume that $\bB$ lies in the relative interior of $K$; otherwise we can restrict ourselves to the proper face $F$ of $K$ in which $\bB$ lies and consider distributions on $\Omega' = \{ \omega \in \Omega : f(\omega) \in F \}$, since any distribution on $\Omega$ satisfying $\E f = \bB$ must be supported on $\Omega'$.   If $F = \{ \bB \}$ then the lemma follows from the fact that the maximum entropy distribution on a finite set is the uniform distribution.  
		
		We can determine the maximum entropy distribution by solving the convex optimization problem
\begin{align*}
&\max - \sum_{\omega \in \Omega} \mu(\omega) \log \mu(\omega)  \\
&\text{subject to } \sum_{\omega \in \Omega} \mu(\omega)f(\omega) = \bB\\
& \quad \sum_{\omega \in \Omega} \mu(\omega) =1 \\
& \quad  \mu(\omega) \ge0 \, \forall \omega \in \Omega  \,.
\end{align*}
The convex dual to this program is 
\begin{align*}
&\min \boldsymbol b  \cdot \bB+ \log \sum_{\omega \in \Omega} e^{- \boldsymbol b \cdot f(\omega)} \\
&\text{where } \boldsymbol b \in \R^d \,.
\end{align*}
Our assumption that $\bB$ is in the relative interior of $K$ means that the primal problem is strictly feasible; i.e. there exists a strictly positive feasible solution $\mu >0$.  Slater's condition (see e.g.~\cite{boyd2004convex}) then guarantees strong duality: the optima of the primal and dual are equal.  This gives an optimal primal solution
		$$\mu(\omega)  = \frac{1}{Z}  e^{- \boldsymbol b  \cdot f(\omega)}   $$
		where $\boldsymbol b \in \R^d$ and $Z$ are chosen  such that $\sum_{\omega} \mu(\omega) =1$ and $\sum \mu(\omega)f(\omega) = \bB$. The normalizing constant $Z$ is called the partition function in statistical mechanics. The existence of such a $\boldsymbol b $  follows from strong duality. 
		
		We then compute
		\begin{align*}
		H(\mu) &= - \sum_{\omega \in \Omega}  \mu(\omega) \log \mu(\omega) \\
		&= \frac{1}{Z} \sum_{\omega \in \Omega}  e^{- \boldsymbol b  \cdot f(\omega)}  \left( \log Z  + \boldsymbol b  \cdot f(\omega) \right )  \\
		&= \log Z +  \boldsymbol b  \cdot \bB \,.
		\end{align*}
		On the other hand, $\mu(\Lam_{\bB}) =  \frac{1}{Z} |\Lam_\bB|  e^{-\boldsymbol b  \cdot \bB}$, and putting these together yields~\eqref{eqEntropyGeneral}. 
	\end{proof}

If the  set $\Omega$ is countably infinite, a maximum entropy distribution subject to a given mean constraint may not exist.  See, e.g.~\cite{csiszar1975divergence,cencov2000statistical}, for some sufficient conditions.  The following lemma will suffice for our application.
\begin{lemma}
		\label{lemEntropyLemma2}
		Let $\Omega$ be a countably infinite set and let $f : \Omega \to \R^d$.   For $\bB \in \R^d$ define $\Lam_{\bB} = \{ \omega \in \Omega: f(\omega) = \bB \}$ and suppose $\Lam_{\bB}$ is finite and non-empty.  Suppose further that there exists some $\boldsymbol b  \in \R^d$ so that  $Z = \sum_{\omega \in \Omega} e^{-\boldsymbol b  \cdot f(\omega)} < \infty$, and with  $\mu(\omega) = \frac{1}{Z} e^{- \boldsymbol b \cdot f(\omega)}$, we have $\E_{\mu} f = \bB$.    Then $\mu$ is the maximum entropy distribution on $\Omega$ so that $\E_{\mu} f =\bB$, and
		\begin{equation}
		\label{eqEntropyGeneral2}
		| \Lam_{\bB}|  = e^{H(\mu)}  \cdot \mu(\Lam_{\bB}) \,.
		\end{equation}
	\end{lemma}
\begin{proof}
The fact that $\mu$ is the maximum entropy distribution follows from the strict convexity of the entropy function.  The calculation of $H(\mu)$ and the verification of~\eqref{eqEntropyGeneral2} then follow exactly as in the proof of Lemma~\ref{lemEntropyLemma}.
\end{proof}	
	
Before applying  Lemma~\ref{lemEntropyLemma2} to our setting, we need a lemma relating Assumption~\ref{Assumption1} to the existence of a solution to a system of equations.

\begin{lemma}
\label{lemDiscreteBetas}
	Suppose there exists $  \bbeta = ({\bbeta}_j)_{j \in J} \in \R^J$ so that $$\int_0^\infty \frac{x^j}{\exp\left(\sum_{i \in J} {\bbeta}_i x^i\right) - 1}\,dx = \balpha_j$$ for all $j \in J$.  Then for $n$ sufficiently large, there exists $\widehat \bbeta = (\widehat \bbeta_j)_{j \in J}$ so that 
\begin{equation}
\label{eq:geo-discrete}
\sum_{k \geq 1} \frac{k^j}{\exp\left(\sum_{i \in J} \widehat\bbeta_i k^i\right)-1} = \lfloor \balpha_j n^{(j+1)/2}\rfloor\,.
\end{equation}
	Further, as $n$ tends to infinity $\bbh_j n^{j/2} \to \bbeta_j$ for each $j \in J$.
\end{lemma}

To prove Lemma~\ref{lemDiscreteBetas} we need the following basic calculus fact.
\begin{lemma}\label{lem:IFT}
	Let $g: \R^d \to \R^d$ be continuously differentiable with $g(0) = 0$ and assume $g'(0) = M$ is invertible.  Suppose there is a $\delta > 0$ so that for $\|x \| \leq \delta$ we have $\|M^{-1}\|\cdot \|g'(x) - M\| \leq 1/2$.  Then for all $y$ with $\|M^{-1}\| \cdot \|y\| \leq \delta/2$ there is some $\|x\| \leq \delta$ so that $g(x) = y$.
\end{lemma}
	\begin{proof}
		Set $x_0 = 0$ and $x_k = M^{-1}\left(M x_{k-1} + y - g(x_{k-1})\right)$ for $k \geq 1$.  Then we first note that for any $x$ with $\|x\| \leq \delta$ we have  \begin{align*}\|M^{-1}(M x - g(x))\| &\leq \| M^{-1}\| \cdot\| M x  - \int_0^1 g'(tx) \cdot x\,dt \| \\
		&\leq \|M^{-1}\| \|x \| \cdot \max_{t \in [0,1]} \| g'(tx)\| \\
		&\leq \delta/2\,.
		\end{align*}
		By induction, we claim that $\|x_k \| \leq \delta$.  Indeed \begin{align*}
		\|x_k\| \leq  \|M^{-1}(M x_{k-1} - g(x_{k-1}))\| + \|M^{-1}\| \|y\|\leq \delta\,.
		\end{align*}
		
		We now want to show that the sequence $\{x_k\}$ is Cauchy.  By the mean-value theorem, we have 
		\begin{align*}\|x_k - x_{k-1}\| &\leq \|M^{-1}\| \|M(x_{k-1} - x_{k-2}) - \left(g(x_{k-1}) - g(x_{k-2})\right)\| \\
		&= \|M^{-1}\| \|M(x_{k-1} - x_{k-2}) - \int_0^1 g'(x_{k-2} + t(x_{k-1} - x_{k-2})) \cdot(x_{k-1} - x_{k-2})\| \\
		&\leq \| M^{-1}\| \max_{t \in [0,1]}\| g'(x_{k-2} + t(x_{k-1} - x_{k-2})) - M \| \cdot \|x_{k-1} - x_{k-2}\|\,. 
		\end{align*}
		We claim that $\|x_k - x_{k-1}\| \leq \delta 2^{-k+1}$ and prove so by induction.  Since $\|x_k\| \leq \delta$ and the ball of radius $\delta$ is convex, we have  $\| x_{k-2} + t(x_{k-1} - x_{k-2})\| \leq \delta$ and so by the above we may bound $$\|x_k - x_{k-1}\| \leq \frac{1}{2} \| x_{k-1} - x_{k-2}\|$$
		completing the proof that $\| x_k - x_{k-1} \| \leq \delta 2^{-k+1}$.  
		
		This shows that $\{x_k\}$ is Cauchy and thus converges to some $x_\infty$. Taking limits of both sides of the definition of $x_k$ then shows that $g(x_\infty) = y$.
	\end{proof}

Now we prove Lemma~\ref{lemDiscreteBetas}.
\begin{proof}[Proof of Lemma \ref{lemDiscreteBetas}]
	Rewrite our desired equality as \begin{equation*}
	f_j(\widehat \bbeta):= n^{-1/2}\sum_{k \geq 1} \frac{(k n^{-1/2})^j}{\exp\left(\sum_{i \in J} (\widehat\bbeta_i n^{i/2}) (k n^{-1/2})^i\right)-1} = n^{-(j+1)/2}\lfloor \balpha_j n^{(j+1)/2}\rfloor\,.
	\end{equation*}
	
	Note that evaluating at $\widehat\bbeta_0 := ({\bbeta}_j n^{-j/2})_{j \in J}$ gives $$f_j(\widehat\bbeta_0 ) \to \balpha_j$$ as $n \to \infty$.  Further, if we define $\mathbf{f} = (f_j)_{j \in J}$ to be a function from $\R^J \to \R^J$ then observe that as $n \to \infty$ we have $$(f'(\widehat\bbeta_0))_{i,j} \to -\int_0^\infty \frac{x^{j+i} \exp\left(\sum_{i \in J} {\bbeta}_i x^i \right) }{\left(\exp\left(\sum_{i \in J} {\bbeta}_ix^i \right)  - 1 \right)^2}\,dx = -\Sigma_{i,j}\,.$$
	Since $\Sigma$ is a Gram matrix of linearly independent entries, it is positive definite and thus invertible.  This means that we can find a $\delta$ so that for $n$ sufficiently large we have $\| f'(\widehat\bbeta_0) ^{-1}\| \cdot \| f'(\widehat\bbeta) - f'(\widehat\bbeta_0)\| \leq 1/2$ for $\|\widehat\bbeta - \widehat\bbeta_0\| \leq \delta$.  For $n$ sufficiently large we have 
$$\|f'(\widehat\bbeta_0)^{-1}\| \cdot \| \left(n^{-(j+1)/2} \lfloor \balpha_j n^{(j+1)/2} \rfloor\right)_{j \in J} - (\balpha_j)_{j \in J}\| \leq \delta/2 \,,$$
 and so we may apply Lemma \ref{lem:IFT} to find the desired solution.  Noting that we may take $\delta \to 0$ slowly shows convergence.
\end{proof}

	As a corollary of Lemmas~\ref{lemEntropyLemma2} and~\ref{lemDiscreteBetas} we derive a formula for $p(\mathbf N)$.
	\begin{cor}
		\label{corPartitionNumberFormula}
		Let $J$ be a profile set and suppose $\balpha  \in \R^J$ satisfies Assumption~\ref{Assumption1}. 
		Let  $\mathbf N =\bN(\balpha,n) $.  Then for large enough $n$, 
		\begin{equation}
		\label{eqPNmu2}
		p(\mathbf N) = e^{H(\mu_n)} \mu_n( \cP(\mathbf N)) \,,
		\end{equation}
		where $\mu_n$ is the maximum entropy distribution on $\cP$ so that   
		\begin{equation}
		\label{eqConstraints1}
		\E_{\lam \sim \mu_n} \sum_{ x \in \lam} x^j = \bN_j 
		\end{equation}
		for all $j \in J$.  In particular, $\mu_n$ is the product measure on $\N_0^\N$ where the projection $\mu_n^k$ to  coordinate $k$ is given by a geometric random variable with parameter $p_k:= 1 - \exp(-\sum_{j \in J} \widehat\bbeta_j k^j )$ where $\widehat\bbeta$ is the solution to~\eqref{eq:geo-discrete} guaranteed by Lemma~\ref{lemDiscreteBetas}.
	\end{cor}

	\begin{proof}
	Recall that the set of all integer partitions, $\cP$, is a countable set. Let $f: \cP \to \R ^{J}$ be defined by $f_j(\lam)  = \sum_{x \in \lam } x^j$.    Since $\balpha$ satisfies Assumption~\ref{Assumption1}, there exists $ \bbeta$ solving~\eqref{eq:geo-cont}, and so by Lemma~\ref{lemDiscreteBetas}, there exists $\widehat\bbeta$ solving the system~\eqref{eq:geo-discrete}.  Let $\mu_n$ be the distribution on $\N_0^\N$ described in the last sentence of the statement of the lemma.  

For a  partition $\lam$, let $a_k$ be the multiplicity of $k$ in $\lam$.  We write $f_j(\lam) = \sum_{k \ge 1} a_k k^j$ and compute
\begin{align}
\nonumber
\log \mu_n(\lam) &= \sum_{k \ge 1}  \log p_k  + a_k \log (1-p_k) \\
\nonumber
&= \sum_{k \ge 1}   \log \left(  1 - \exp(-\sum_{j \in J} \widehat\bbeta_j k^j )  \right)  - \sum_{k \ge 1} a_k   \sum_{j \in J} \widehat\bbeta_j k^j \\
\label{eqlogmu}
&= - \log Z -  \sum_{j \in J}\widehat\bbeta_j  \sum_{k \ge 1} a_k   k^j
\end{align}
where we have defined $\log Z := - \sum_{k \ge 1}   \log \left(  1 - \exp(-\sum_{j \in J} \widehat\bbeta_j k^j )  \right)$.   Exponentiating~\eqref{eqlogmu} gives $$\mu_n (\lam) = \frac{1}{Z} e^{-\widehat\bbeta \cdot f(\lam) }\,,$$
as required for Lemma~\ref{lemEntropyLemma2}.
  Moreover, since $f_j$ is additive over coordinates for each $j \in J$,
\begin{align*}
\E_{\lam \sim \mu_n} [ f_j (\lam)] &=  \sum_{k \ge 1}  \frac{ \exp(-\sum_{\ell \in J} \widehat\bbeta_\ell k^\ell )  }{ 1- \exp(-\sum_{\ell \in J} \widehat\bbeta_\ell k^\ell )    }  k^j \\
&= \sum_{k \ge 1}  \frac{k^j }{  \exp(\sum_{\ell \in J} \widehat\bbeta_\ell k^\ell )  -1  }   = N_j
\end{align*}
by our assumption on $\widehat\bbeta$.   Therefore by Lemma~\ref{lemEntropyLemma2}, $\mu_n$ is the maximum entropy distribution on $\cP$ such that $\E_{\lam \sim \mu_n} \sum_{x \in \lam} x^j = N_j$ for $j \in J$ and $p(\mathbf N) = e^{H(\mu_n)} \mu_n( \cP(\mathbf N))$.
	\end{proof}

Finally we determine the optimum and optimizer of the continuous convex program~\eqref{eqMbsb}. 	
\begin{lemma}
\label{lemContinuousSolution}
Suppose $\balpha$ satisfies Assumption~\ref{Assumption1} and let $\bbeta$ be a solution to~\eqref{eq:geo-cont}.   Let
\begin{align*}
f^*(x) &= \frac{1}{ \exp( \sum_{j \in J} \bbeta_j x^j) - 1} \,.
\end{align*}
Then $f^*$ is an optimal solution to~\eqref{eqMbsb}, and thus 
$$M(\balpha) = \int_{0}^\infty G( f^*(x)) \, dx \, . $$
\end{lemma}
\begin{proof}
This is a consequence of duality in infinite dimensional convex programming (see e.g.~\cite{rockafellar1974conjugate}).   The program~\eqref{eqMbsb} is an infinite-dimensional convex program since the constraints are linear in $f(x)$ and the objective function $\int_{0}^\infty G(f(x)) \, dx $ is strictly concave (this follows from the fact that $G(\eta)$ is a strictly concave function of $\eta$).   The Lagrangian associated to this program is
\begin{align*}
L( f, y, \bbeta) &= \int_{0}^\infty G(f(x)) \, dx - \sum_{j \in J} \bbeta_j  \left( \int x^j f(x) \, dx - \balpha_j \right)\,.
\end{align*}
A sufficient condition for optimality of $f^*$ is that $f^*$ is feasible and  the function derivative of $L(f,y, \bbeta)$ with respect to $f$ vanishes at $f^*$.  This is the condition
\begin{align*}
0= \log \left (1+ \frac{1}{f(x)} \right) - \sum_{j \in J} \bbeta_j x^j
\end{align*}
 for all $x \in [0, \infty)$.   This is satisfied by $f^*(x)$ with $ \bbeta$ as given by Assumption~\ref{Assumption1} since $ \log \left (1+ \frac{1}{f(x)} \right) = \sum_{j \in J} \bbeta_j x^j $, and so $f^*$ is an optimal solution. 
\end{proof}

	\section{Asymptotics of $e^{H(\mu_n)}$}
	\label{secOPT}
	
	In this section we  compute the asymptotics of the first term in the formula~\eqref{eqPNmu2}.    In what follows we fix the profile set $J$ and $\balpha \in \R_+^J$ satisfying Assumption~\ref{Assumption1}.  We let $\mathbf N = \bN(\balpha,n)$ and let $\bbeta$ be the solution to~\eqref{eq:geo-cont} guaranteed by Assumption~\ref{Assumption1}.

	\begin{lemma}
		\label{lemEntropyApprox}
		With $\mu_n$ defined as in Corollary~\ref{corPartitionNumberFormula},  we have 
		\begin{equation}
		H(\mu_n) =  \sqrt{n} M(\balpha) + b_1(J) \log n +c_1(\balpha) +o(1) 
		\end{equation}
		as $n \to \infty$, where $b_1(J) =  - \frac{\jmin}{4}$ and 
\begin{align*}
c_1(\balpha) &= -\frac{\jmin}{2} \log(2\pi)    + \frac{\one_{\jmin = 0}}{2}\left(\frac{{\bbeta}_0}{e^{{\bbeta}_0} - 1} - G\left ( \frac{ 1} { e^{{\bbeta}_0}-1} \right)\right) + \frac{\one_{\jmin \geq 1}}{2} \log({\bbeta}_{\jmin})    \,.
\end{align*}
	\end{lemma}
	
Let $\widehat\bbeta$ be as in Lemma~\ref{lemDiscreteBetas}.  For $x > 0$ let $ f^* (x) =  (\exp\left ( \sum_{j \in J} \bbeta_j x^j\right) - 1)^{-1}$ and for $k \ge 1$ let $\widehat f(k) = (\exp\left( \sum_{j \in J} \widehat\bbeta_j k^j\right) - 1)^{-1} $.  Then from Lemma~\ref{lemContinuousSolution} and Corollary~\ref{corPartitionNumberFormula}, we have 
\begin{align*}
M(\balpha) &= \int _{0}^\infty G( f^* (x)) \, dx \, ,\\
H(\mu_n) &= \sum_{k \ge 1} G(\widehat f(k))  \,.
\end{align*}

We start by relating the parameters $\bbh_j$ to their analogues $\bbeta_j$.  
\begin{lemma}\label{lem:eps}
	Define $\bm \eps$ via $\bbh_j n^{j/2} = \bbeta_j + \beps_j / \sqrt{n}$.  Then $$\beps = \Sigma^{-1}\left(-\frac{\mathbf{v}_{j = \jmin} }{2}\left(\frac{\one_{\jmin = 0}}{e^{\bbeta_0} - 1} + \frac{\one_{\jmin \geq 1}}{\bbeta_{\jmin}} \right)  \right)^T + O(n^{-1/2})$$
	where $\mathbf{v}_{j = \jmin}$ is the vector with $1$  in the $j = \jmin$ coordinate and $0$ in all other coordinates.
\end{lemma}
\begin{proof}
	First write \begin{align*}
	\balpha_j &= n^{-1/2} \sum_{k \geq 1} \frac{(k n^{-1/2})^j}{\exp\left(\sum_{i \in J} (\bbh_i n^{i/2}) (k/\sqrt{n})^i \right) - 1   }  \\ 
	&= - \frac{\one_{j = \jmin}}{2 \sqrt{n}} \left(\left(\frac{\one_{\jmin = 0}}{e^{\bbh_0} - 1} \right) + \frac{\one_{\jmin \geq 1}}{\bbh_{\jmin} n^{\jmin/2}} \right) + \int_0^\infty \frac{x^j}{\exp\left( \sum_{i \in J} (\bbh_i n^{i/2}) x^i \right) - 1}\,dx + O(1/n)\,.
	\end{align*}
	
	Taylor expanding the integrand at $\bbeta_i$ gives 
	\begin{align*}\int_0^\infty \frac{x^j}{\exp\left( \sum_{i \in J} (\bbh_i n^{i/2}) x^i \right) - 1}\,dx &= \balpha_j - \sum_{i \in J} \frac{\beps_i}{\sqrt{n}} \int_0^\infty x^{i + j} \frac{\exp\left(\sum_{k \in J}(\bbh_k n^{k/2}) x^k  \right)}{\left( \exp\left(\sum_{k \in J}(\bbh_k n^{k/2}) x^k  \right) - 1\right)^2}\,dx + O(1/n) \\
	&= \balpha_j -  \sum_{i \in J}(1 + o(1)) \frac{\beps_i}{\sqrt{n}} \Sigma_{i,j} + O(1/n)
	\end{align*}
	and so \begin{align*}
	\sum_{i \in J} (1 + o(1)) \beps_i \Sigma_{i,j} &= - \frac{\one_{j=\jmin}}{2}\left(\left(\frac{\one_{\jmin = 0}}{e^{\bbh_0} - 1} \right) + \frac{\one_{\jmin \geq 1}}{\bbh_{\jmin} n^{\jmin/2}} \right) + O(n^{-1/2}) \\
	&= - \frac{\one_{j=\jmin}}{2}\left(\left(\frac{\one_{\jmin = 0}}{e^{\bbeta_0} - 1} \right) + \frac{\one_{\jmin \geq 1}}{\bbeta_{\jmin}} \right)+ O(n^{-1/2})\,.
	\end{align*}
	Solving for $\beps$ completes the Lemma.	
\end{proof}

	We need an Euler-Maclaurin summation calculation that we postpone to the Appendix~\ref{secEulerMac}.
	\begin{lemma}\label{lem:EM-asymptotic}
		For $ \bsg \in \R^J$ with $\sum_{j \in J} \bsg_j t^j \geq 0$ for all $t \geq0$ and $\bsg_{\jmin} > 0$,  we have as $t \to 0^+$, 
		\begin{align*}
		\sum_{k \geq 1}  G \left(  \frac{1}{    \exp ( \sum_{j\in J}  \bsg_j (tk)^j    )  -1     } \right) &= t^{-1}\int_0^\infty G  \left(   \frac{1}{    \exp ( \sum_{j\in J}  \bsg_j x^j    )  -1     } \right)\,dx - \frac{\jmin}{2}\log(2\pi/ t) \\
		&- \frac{\one_{\jmin = 0}}{2}G \left(\frac{1}{  e^{\bsg_0} -1}\right) 
		+ \frac{\one_{\jmin \geq 1}}{2}\left( \log \bsg_{\jmin} - 1 \right) + o(1)  
		\end{align*}
		where the error is uniform for $ \bsg$ in a compact set  $K \subset \R^J$ satisfying the hypotheses.

	\end{lemma}
	Using this we compute the asymptotics of $H(\mu_n)$.

	\begin{proof}[Proof of Lemma~\ref{lemEntropyApprox}] 
		Write 
		\begin{align*}
H(\mu_n) &= \sum_{k \geq 1} G \left(  \frac{1}{ \exp\left( \sum_{j \in J} \widehat\bbeta_j k^j\right) - 1} \right) \\
&= \sum_{k \geq 1} G\left( \frac{1}{  \exp\left(\sum_{j \in J} (\widehat \bbeta_j n^{j/2})(k n^{-1/2})^j \right)  -1 } \right)
		\end{align*}

		and apply Lemma \ref{lem:EM-asymptotic} to $\bsg = (\widehat\bbeta_j n^{j/2})_{j \in J}$ with $t = n^{-1/2}$ to obtain
		\begin{align}
		\label{eqHmunId1}
		H(\mu_n) &= \sqrt{n} \int_0^\infty G\left( \frac{1}{  \exp\left(\sum_{j \in J} (\widehat \bbeta_j n^{j/2}) x^j \right) -1}  \right)\,dx - \frac{\jmin}{2}\log(2\pi n^{1/2})  \\ 
		\nonumber
		&\qquad- \frac{\one_{\jmin = 0}}{2} G \left (\frac{1}{  e^{\widehat\bbeta_0} -1} \right)  + \frac{\one_{\jmin \geq 1}}{2}\left( \log (\widehat\bbeta_{\jmin} n^{\jmin/2}) - 1 \right) + o(1)\,.
		\end{align}

Now recall $\widehat \bbeta_j n^{j/2} = \bbeta_j + \bm{\eps}_j n^{-1/2}$, then Lemma \ref{lem:eps} states 
	$$\beps = \Sigma^{-1}\left(-\frac{\mathbf{v}_{j = \jmin} }{2}\left(\frac{\one_{\jmin = 0}}{e^{\bbeta_0} - 1} + \frac{\one_{\jmin \geq 1}}{\bbeta_{\jmin}} \right)  \right)^T + O(n^{-1/2})$$
	where $\mathbf{v}_{j = \jmin} \in \R^J$ is the vector with a $1$ in the $j = \jmin$ coordinate and zeros elsewhere, and $\Sigma$ given in~\eqref{eqSigmaDef}.  Then we have 
		\begin{align}
		\label{eqIntId2}
		 \int_0^\infty G\left( \frac{1}{  \exp\left(\sum_{j \in J} (\widehat \bbeta_j n^{j/2}) x^j \right) -1}  \right)\,dx &=  \int_0^\infty G\left( \frac{1}{  \exp\left(\sum_{j \in J} ( \bbeta_j + \bm{\eps}_j n^{-1/2}) x^j \right) -1}  \right)\,dx     \\
		 &= M(\balpha)- n^{-1/2}(\bm{\eps} {\Sigma} {\bbeta}) + o(n^{-1/2}) \\
		\nonumber
		&= M(\balpha) + \frac{1}{2n^{1/2}}\left(\one_{\jmin = 0}\left(\frac{{\bbeta}_0}{e^{{\bbeta}_0} - 1}  \right)  + \one_{\jmin \geq 1}  \right)+ o(n^{-1/2})\,.
		\end{align}
		Putting~\eqref{eqHmunId1} and~\eqref{eqIntId2} together gives the lemma.
	\end{proof}

	\section{The Local CLT}
	\label{secCLT}
	In the previous two sections, we gave an exact formula~{\eqref{eqPNmu2}} for $p(\mathbf N)$ and computed the asymptotics of its first factor, $e^{H(\mu_n)}$. In this section, we compute the asymptotics of the second factor, $\mu_n(\mathcal P(\mathbf N))$, the probability a random partition generated by the entropy-maximizing distribution $\mu_n$ has profile $\mathbf N$.
	\begin{lemma}
		\label{lemCLTapprox}
		For all $\balpha$ satisfying Assumption~\ref{Assumption1} and $n$-feasible,
		\begin{equation}
		\mu_n (\cP(\mathbf N)) = (1+o(1))  \cdot \frac{|\QQ_J|}{(2\pi)^{|J|/2} \det (\Sigma)^{1/2}}  \cdot n^{-\sum_{j \in J}(j/2 + 1/4) }
		\end{equation}
		as $n \to \infty$,
		where $\mathbf N = (\lfloor \balpha_j n^{(j+1)/2} \rfloor)_{j \in J} $.  Further, the error is uniform for $\balpha$ varying in a compact set of $\R^J$ satisfying Assumption~\ref{Assumption1}.
	\end{lemma}
	
	To prove Lemma~\ref{lemCLTapprox} we prove  a multivariate local central limit theorem.  While the proof of this local central limit theorem can be easily adapted to a broader setting---such as for random variables other than geometrics---we state it for the random variables of interest.  
	\begin{theorem}\label{th:LCLT}  
		Fix a compact set $K \subset \R^{J}$.  Suppose $\{Y_k\}_{k \geq 1}$ are independent geometric random variables with parameters $p_k$ where $p_k = 1 - \exp\left(-\sum_{j \in J} \bbh_j k^j \right)$ and $\bbh = (\bbh_j)_{j \in J}$ with $(\bbh_j n^{j/2})_{j \in J} \in K$ for all $n$.
		Define $$\bX = \left(\sum_{k \geq 1} Y_k k^j\right)_{j \in J}   $$ and set $S= \Cov(\bX)$.  Then \begin{align*}
		\sup_{\ba}\left|\P(\bX= \ba) -  \frac{\one_{\ba \in \NT}|\QQ_J|}{\sqrt{(2\pi)^{|J|}\det(S) } } \exp\left(-\frac{1}{2}(\mathbf{a} - \E \bX)^T (S)^{-1} (\mathbf{a} - \E \bX)\right)  \right| &= o(\det(S)^{-1/2}) \\
		&=o(n^{- \sum_{j \in J}(j/2 + 1/4) })
		\end{align*}
		where the error term depends only on $K$ and $J$.
	\end{theorem}

	Before discussing the proof of this theorem, we will attempt to give some intuition about the statement itself. First, we may think of $\mathbf X$ as the profile of a random partition sampled from the maximum entropy distribution $\mu_n$ (where $Y_k$ is the number of times $k$ appears as a part in the random partition). Our goal in this section is to estimate the probability that $\mathbf X$ is equal to its mean, $\E \mathbf X = \mathbf N$ (Lemma~\ref{lemCLTapprox}). In fact, the theorem above estimates the probability mass function $\P(\mathbf X = \mathbf a)$ for \textit{each} possible profile $\mathbf a$, and not only for $\mathbf a = \mathbf N$ (although the multiplicative error given by this estimate is large when $\mathbf a$ is far from $\mathbf N$). We recover Lemma~\ref{lemCLTapprox} by taking~$\mathbf a = \mathbf N$ and noting $\det (S) \sim n^{\sum_{j \in J}(j/2 + 1/4)} \det (\Sigma)$ as $n \to \infty$, as detailed below. 

	The conclusion of the theorem approximates the probability mass function $\P(\mathbf X = \mathbf a)$ in terms of the density of a multivariate Gaussian. Indeed, the expression above is very nearly the density of a $|J|$-dimensional Gaussian with mean $\E \bX$ and covariance matrix $S$; the only difference is the factor $\one_{\ba \in \NT} |\QQ_J|$.

	We now explain the factor $\one_{\ba \in \NT} |\QQ_J|$.  The random variable $\bX$ is constrained by  number-theoretic identities.  For instance, for any integer $k$  we have $k^2 \equiv k \mod 2$ and so if $1, 2 \in J$ then we must have that $X_1 \equiv X_2 \mod 2$.  The set $\NT$  describes these constraints. While the vector $\bX$ is close to a Gaussian when centered and scaled, its support is a subset of this smaller set $\NT$; this means that the atoms of $\bX$ must be assigned larger probability by some quantity roughly reflecting the density of the set $\NT$ in $\Z^J$.  This density is precisely $1/|\QQ_J|$, thus giving the factor $|\QQ_J|$.

	We  will show that the rescaled random variable defined via $$\widehat{\bX} = \left(n^{-j/2-1/4} \sum_{k\geq 1} (Y_k - \E Y_k)k^j \right)_{j \in J}$$
	converges in distribution to a centered multivariate Gaussian provided we have $(\bbeta_j n^{j/2})_{j \in J} \to \balpha$.  This is a consequence of \eqref{eq:V}.
	
	Before proving Theorem~\ref{th:LCLT}, we show that Lemma~\ref{lemCLTapprox} follows from Theorem~\ref{th:LCLT}.
	
	\begin{proof}[Proof of Lemma \ref{lemCLTapprox}]
		Apply Theorem \ref{th:LCLT} with the $\bbh$ guaranteed by Lemma~\ref{lemDiscreteBetas}.  By hypothesis, we  have  $\bN = \E \bX$.  Taking $\ba = \bN$ then yields $$\mu_n(\mathcal{P}(\bN)) = \P(\bX = \bN) = (1 + o(1)) \frac{|\QQ_J|}{(2\pi)^{|J|/2} \det S^{1/2} }\,.$$
	since $\balpha$ is $n$-feasible.  Recalling the definition of $\Sigma$ from \eqref{eqSigmaDef} and that $\bbh_j n^{j/2} \to \bbeta_j$ for each $j$, the computation \begin{align*} 
	S_{i,j} &= \sum_{k \geq 0} k^{i+j} \frac{\exp\left(\sum_{\ell \in J} \bbh_\ell k^\ell \right)  }{\left(\exp\left(\sum_{\ell \in J} \bbh_\ell k^\ell \right)-1 \right)^2} \\
	&= n^{(i + j + 1)/2}	\left( \frac{1}{\sqrt{n}}\sum_{k \geq 0} (k/\sqrt{n})^{i+j} \frac{\exp\left(\sum_{\ell \in J} \bbh_\ell n^{\ell/2} (k/\sqrt{n})^\ell \right)  }{\left(\exp\left(\sum_{\ell \in J} \bbh_\ell n^{\ell/2} (k/\sqrt{n})^\ell \right)-1 \right)^2}\right) \\
	&= n^{(i + j + 1)/2} \Sigma_{i,j}	(1 + o(1)) 
	\end{align*}	
	shows $\det S =n^{\sum_{j \in J}(j + 1/2)} \det \Sigma  (1 + o(1)) $, which completes the proof.
	\end{proof}

	\subsection{Outline and Preliminaries}
	Theorem \ref{th:LCLT} is proved  via Fourier analysis.  Specifically, for $\bt \in \R^J$ and  $x \in \R$, define $Q_{\bt}(x) := \sum_{j \in J} t_j x^j$ and set $\varphi_k(t) := \E e^{i t Y_k}$.  
	
	The characteristic function $\varphi(\bt) := \E e^{i \langle \bt, \bX \rangle}$ may then be written as \begin{equation}\label{eq:product}
	\varphi(2\pi \bt) = \prod_{k \geq 1} \varphi_k(2\pi Q_{\bt}(k))\,.
	\end{equation}
	Fourier inversion (e.g.~\cite[Theorem XV.3.4]{feller}) gives \begin{equation}\label{eq:prob-inversion}
	\P(\bX = \ba) = \int_{(-1/2,1/2]^{J}}\varphi(2\pi\bt)e^{-i \langle 2\pi \bt,\ba \rangle}\,d\bt\,. 
	\end{equation}
	
	Theorem \ref{th:LCLT} will be proven by analyzing the integral in \eqref{eq:prob-inversion}.  Our method is a variant on the Hardy-Littlewood circle method.  The set $(-1/2,1/2]^J$ is broken up into \emph{major} and \emph{minor} arcs: the major arcs are sets of small volume that contribute the bulk of the mass of the integral in \eqref{eq:prob-inversion} while the minor arcs are the rest.  
	
	When proving a local central limit theorem, it is often the case that the only major arc is a neighborhood of $\bt = 0$.  In our problem, however, this does not occur; the number-theoretic obstructions force $\P(\bX = \ba) = 0$ for $\ba \notin \NT$.
	On the Fourier side, this means that there are multiple major arcs, and in the case of $\ba \notin \NT$, the integral over them cancels out.  
	
	Before beginning the analysis, we give an outline of the proof of Theorem \ref{th:LCLT}: the general flow is that the size of the set we are integrating over decreases as the proof goes on. First, for each $\delta > 0$, we define the following small neighborhood around 0: $$U = U(\delta) = \{\bt \in (-1/2,1/2]^J : |t_j| \leq \delta n^{-j/2} \text{ for all }j \in J \}.$$  We then define $$R(\delta) = (-1/2,1/2]^J \setminus \left(\bigcup_{q \in \QQ_J} (q + U) \right),$$
	where for each polynomial $q(z) \in \sum q_j z^j \in \QQ_J$, we interpret $q = (q_j)_{j \in J}$ as a vector in $(-1/2,1/2]^J$.  We show that for each $R(\delta)$, the integral on $R(\delta)$ is exponentially small.  In the language of the circle method, this states that the major arcs are exactly the neighborhoods $q + U$ for $q \in \QQ_J$. This is carried out in Section~{\ref{sec:minor-arcs}}, and a more detailed summary is given there, including an intuitive description of how the polynomials in $\QQ_J$ arise.

	Next, in Section \ref{sec:combined} we show that the integral over each neighborhood $q + U$ is equal provided $\ba \in \NT$, and so we may combine the $|\QQ_J|$-many integrals over $\bigcup_{q \in \QQ_J} (q+ U)$ into $|\QQ_J|$ times the integral over $U$.   
	
	Section \ref{subsecApproxOverU} compares the integral of our characteristic function over $U$ to the integral over $\R^d$ of the characteristic function of the corresponding Gaussian by showing upper and lower bounds on the matrix $S$ when viewed as a quadratic form (Lemma \ref{lem:Sigma-scale}).
	
	Section \ref{subsecU-minus-V} reduces our integral even further to the set $V = \{\bt : |t_j| \leq (\log n)n^{-j/2 - 1/4}\}$ by comparing the characteristic function of each geometric variable of bounded mean to the characteristic function of the corresponding Gaussian (Lemma \ref{lem:cumulant-bound}); while not all of our geometric variables have bounded mean, the bulk of the contribution to the covariance matrix $S$ comes from the parameters $Y_k$ of bounded mean (Lemma \ref{lem:truncate-var}), which is sufficient for this purpose.
	
	Finally, {\ref{subsecV}} evaluates the integral over $V$; this is equivalent to an ordinary multivariate central limit theorem with more careful tracking of error terms.

	\subsection{Bounding the minor arcs: reducing to a neighborhood of $\QQ_J$} \label{sec:minor-arcs}
	
	Recall that our goal is to estimate the following integral
	
	\begin{equation*}
		\P(\bX = \ba) = \int_{(-1/2,1/2]^{J}}\Bigg(\prod_{k \geq 1} \varphi_k(2\pi Q_{\bt}(k))\Bigg)\,e^{-i \langle 2\pi \bt,\ba \rangle}\,d\bt
	\end{equation*}
    (obtained by combining \eqref{eq:product} and \eqref{eq:prob-inversion}).	To help identify the major arcs, we will see in Lemma \ref{lem:char-bound} we have that if $ t$ is not close to an integer, then $|\varphi_k(2\pi t)|$ is uniformly bounded away from $1$.  This shows that if there are many integer values $k$ for which the polynomial $Q_{\bt}(k)$ is not near an integer, then the integrand above is small.
	
	This means that we have to understand when the polynomial $Q_{\bt}(x)  = \sum_{j \in J} x^{j} t_j$ is close to an integer-valued polynomial, i.e.\ a polynomial $q$ so that $q(\Z) \subset \Z$.  Perhaps surprisingly, there are many such polynomials, even if we omit those with integer coefficients; as an example, the binomial coefficients $\binom{z}{k}$ are integer-valued polynomials but do not have integer coefficients.  Motivated by this, for a given $J$, recall that we define 
	$$\QQ_J = \left\{ \sum_{j \in J} t_j x^j :  \sum_{j \in J} t_j m^j \in \Z \text{ for all }m \in \Z, t_j \in (-1/2,1/2]  \right\} $$ 
	to be the set of integer-valued polynomials of interest.
	
	P\'olya \cite{polya1915ganzwertige} showed that all integer-valued polynomials are integer linear combinations of binomial coefficients; this may be proved by induction on the degree and examining finite differences.  This shows, for instance, that $|\QQ_J|$ is finite.  To better understand $\QQ_J$, we look at an extreme case: \begin{lemma} \label{lem:qd-count}
		For each $d \geq 1$, $|\QQ_{[d]}| = \prod_{j = 1}^d (j!)$.
	\end{lemma}
	\begin{proof}
		Since every element of $\R/\Z$ has a unique representative in $(-1/2,1/2]$, it is sufficient to count the number of polynomials of degree at most $d$ in $\R/\Z$ that are integer-valued.  Each such equivalent class is equal to $$\sum_{j = 1}^d m_j \binom{z}{j}$$ for some choice of integers $m_j$.  For each selection of $(m_1,\ldots,m_d)$ with $1 \leq m_j \leq j!$ we obtain a distinct representative. Conversely, each integer polynomial may be written in the above form for \emph{some} integers $(m_j)$;  further, for two integers $n_1$ and $n_2$, we have $n_1 \binom{z}{j} \equiv n_2 \binom{z}{j}$ as polynomials with coefficients in $\R/\Z$ if and only if $n_1 \equiv n_2 \mod j!$.  This means that we may uniquely choose the representatives $(m_j)$ to satisfy $1 \leq m_j \leq j!$ for each $j$.
	\end{proof}
	
	Now, recall our definitions of $U$ and $R$: 
	for $\delta > 0$, define $U = U(\delta) = \{\bt \in (-1/2,1/2]^{J} : |t_j| \leq \delta n^{-j/2} \text{ for all }j \in J \}$ and 
	 $$R(\delta) = (-1/2,1/2]^{J} \setminus \left(\bigcup_{q \in \QQ_J} (q + U) \right)\,.$$
	Thus $R(\delta)$ is the set of points that are far from the coefficients of any integer-valued polynomial. Our goal for this section is to show that the contribution of $R(\delta)$ is negligible for \emph{any} choice of $\delta > 0$.  
	 
	 \begin{lemma}\label{lem:R_n} In the context of Theorem \ref{th:LCLT}, for each $\delta > 0$ there exists a constant $c = c(\delta,K)$ so that $$\int_{R(\delta)}| \varphi(2\pi \bt) | \,d\mathbf{t}   = O( e^{-c\sqrt{n}} )\,. $$
	 \end{lemma}
	 
	 Lemma \ref{lem:R_n} will be accomplished by a pointwise bound on $|\varphi(2\pi\bt)|$ on the set $R(\delta)$.   In light of the infinite product in \eqref{eq:product}, in order to show that $|\varphi(2\pi \bt)|$ is exponentially small in $\sqrt{n}$, it is enough to show that on the order of $\sqrt{n}$ many $|\varphi_k(2\pi Q_{\bt}(k))|$ are uniformly less than $1$.  The following elementary fact is a step in this direction.  For a real number $t \in \R$ define $\| t \|_{\R/\Z} = \min_{z \in \Z}|t - z|$ to be the distance to the nearest integer.  We want to show that for if $\| t \|_{\R/\Z}$ is large, then $| \E e^{i t Y}|$ is bounded away from $1$.
	\begin{lemma}	\label{lem:char-bound}
		Fix $\eps > 0$, and let $p \in [\eps,1-\eps]$.  Suppose $Y$ is
		a geometric variable with parameter $p$.  Then for each $\delta_1 > 0$ there exists a $\delta_2 > 0$ so that if $\| t \|_{\R/\Z} \geq \delta_1$ then $| \E e^{i 2 \pi t Y} | \leq 1 - \delta_2$.
	\end{lemma}
	\begin{proof}
		Since $Y$ is integer-valued, we may assume without loss of generality that $t \in (-1/2,1/2]$.  In each case we may uniformly bound the modulus of the characteristic function using compactness and continuity.
	\end{proof}
	
	From Lemma \ref{lem:char-bound} we extract the following simple consequence, which is the engine behind the proof of Lemma \ref{lem:R_n}.
	
	\begin{lemma} \label{lem:engine}
		In the context of Theorem \ref{th:LCLT}, for each $\eps_1, \eps_2 > 0$ there exists a $c > 0$ so that the following holds:
		if $\bt \in \R^J$ satisfies $$\frac{|\{ m \in [\sqrt{n}] : \|Q_{\bt}(m)\|_{\R/\Z} \geq \eps_1   \}|}{\sqrt{n}} \geq \eps_2$$
		then $$|\varphi(2\pi \bt)| \leq e^{-c \sqrt{n}}$$ where $c$ may be chosen uniformly depending only on $K$, $\eps_1$ and $\eps_2$.
	\end{lemma}
	\begin{proof}
		There exists an $\eps_3 > 0$ so that for $k \in [\eps_2 \sqrt{n}/2, \sqrt{n}]$, the geometric parameters $p_k$ lie in the interval $[\eps_3,1 - \eps_3]$.  Thus, for at least $\eps_2 \sqrt{n}/2$ values of $k \in [\eps_2 \sqrt{n}/2, \sqrt{n}]$ we have $$|\varphi_k(2\pi Q_\bt(k))| \leq 1 - \eps_4$$ for some $\eps_4 > 0$ depending only on $\eps_1$ and $\eps_3$.  We then may bound \begin{align*}
		|\varphi(2\pi \bt)| \leq \prod_{k = \eps_2 \sqrt{n}/2}^{\sqrt{n}} |\varphi_k(2\pi Q_\bt(k))| \leq (1 - \eps_4)^{\eps_2 \sqrt{n}/2} \leq e^{- \eps_2 \eps_4 \sqrt{n} / 2}\,.
		\end{align*}
	\end{proof}
	
	We now need a structural result which will say that either $\|Q_{\bt}(x)\|_{\R/\Z}$ is either bounded below quite often, or $Q_{\bt}$ is close to an element of $\QQ_J$.   A quantitative equidistribution theorem of Green and Tao \cite[Proposition 4.3]{green2012quantitative} will make explicit that these are the only two cases.
	 \begin{theorem}[Green-Tao]\label{th:Green-Tao}
		Let $d \geq 0$ and suppose that $g$ is a polynomial with real coefficients of degree $d$.  Suppose that $\delta \in (0,1/2)$. Then either $(g(x) \mod \Z)_{x \in [N]}$ is $\delta$-equidistributed or else there is an integer $k$ satisfying $1 \leq k \ll \delta^{-O_d(1)}$ so that $\| kg \mod \Z \|_{C^\infty[N]} \ll \delta^{-O_d(1)}$.
	\end{theorem}
	
	Some definitions are in order.
	 \begin{defn}
		Let $g$ be a polynomial of degree $d$.  Then there exist unique $s_j$ so that $$g(m) = s_0 + s_1 \binom{m}{1} + \cdots + s_d \binom{m}{d}$$ for each $m$.  Define $$\| g \|_{C^\infty(N)} := \sup_{1 \leq j \leq d} N^{j} \| s_j \|_{\R/\Z}$$ where $\| x \|_{\R/\Z}$ is the nearest distance from $x$ to an integer.
		
		Further, a sequence $\{g(m)\}_{m \in [N]}$ is $\delta$-equidistributed if for all Lipschitz functions $F: \R/\Z \to \mathbb{C}$ and arithmetic progressions $P \subset [N]$ with $|P| \geq \delta N$ we have \begin{align*}
		\left|\frac{1}{|P|}\sum_{m \in P} F(g(m)) - \int_{\R/\Z} F(x)\,dx \right| \leq \delta \| F \|_{\mathrm{Lip}}\,.
		\end{align*} 
	\end{defn}
	
	Theorem \ref{th:Green-Tao} is proven via an effective version of Weyl's equidistribution theorem together with iterating the van der Corpet difference trick.
	Anticipating an application of Theorem \ref{th:Green-Tao}, we first translate the two possibilities of its dichotomy into our setting, beginning with the structured case:
	\begin{lemma}\label{lem:rational-approx}
		Let $g(z) = t_0 + t_1z + \cdots + t_d z^d$ be of degree $d$ and have coefficients in $(-k/2,k/2]$.  There exists a constant $C$ so that if $\| g \|_{C^\infty[N]} \leq R$ then there is an integer-valued polynomial $q(z) = \sum q_j z^j$ so that $\|t_j - q_j\|_{\R/(k\Z)} \leq C R N^{-j}$ for all $j \in [d]$.
	\end{lemma}
	\begin{proof}
		By assumption we have $ N^j \| s_j \|_{\R  / \Z} \leq R$ for all $1 \leq j \leq d$ where $s_j$ is defined by $g(n) = \sum s_i \binom{n}{i}$.  Thus there are integers $m_j$ so that $s_i = m_j + \eps_j$ where $|\eps_j| \leq N^{-j}R$.  Define the linear transformation $T$ to be the map that takes as input $(t_0,\ldots,t_d)$ and outputs $(s_0,\ldots,s_d)$ defined by $$\sum_{j = 0}^d t_j z^j = \sum_{j = 0}^d s_j \binom{z}{j}\,.$$
		
		Since $T$ is linear and invertible, there is a constant $C$ so that $C^{-1} \leq \frac{\| Tx \|}{\| x \|} \leq C$ for all $x \neq 0$.  For a given vector of integers $(m_j)_{j =0}^d$, define $q_j$ to be $T^{-1}((m_j)_j)$ so that $$ \sum_{j = 0}^d q_j z^j = \sum_{j =0}^d m_j \binom{z}{j}\,,$$ and note that the above is an integer-valued polynomial.  By linearity together with the fact that $q_j$ depends only on $s_j,s_{j+1},\ldots,s_d$, we have that 
		$$ T^{-1} \{ x : |x_j - m_j| \leq \eps_j    \} \subset \{ y:  |y_j - q_j| \leq C\sum_{i = j}^d \eps_i   \}\,.$$
		By taking the rational numbers $q_j$ modulo $k \Z$ and replacing $| \cdot |$ with $\| \cdot \|_{\R/(k\Z)}$, we have 
		Thus, we have that $$\|t_j - q_j\|_{\R/(k\Z)}  \leq C \sum_{i = j}^d \eps_i \leq 2 C R N^{-j}.$$
	\end{proof}

	We now show that the Green-Tao definition of equidistributed is good enough for the case at hand.
	\begin{lemma}\label{lem:approx-equid}
		If a sequence $\{ g(m)\}_{m \in [N]}$ is $\frac{1}{32}$-equidistributed then $$
		\frac{|\{m \in [N] : \| g(m)\|_{\R/\Z} \geq 1/4  \} |}{N} \geq \frac{1}{8} \,.$$
	\end{lemma}
	\begin{proof}
		Let $F$ denote the Lipschitz function on $\R/\Z$ that is piecewise linear with $F(0) = F(1/4) = F(3/4) = F(1) = 0$ and $F(1/2) = 1$. Then $\int_{\R/\Z} F(x)\,dx = \frac{1}{4}$, $\| F\|_{\mathrm{Lip}} = 4$ and $F(x) \leq \one_{\| x \|_{\R/\Z} \geq 1/4}$.  By the definition of $1/32$-equidistributed we have $$\left|\frac{1}{N}\sum_{m \in [N]} F(g(m)) - \frac{1}{4} \right| \leq \frac{1}{8}$$
		implying $$\frac{1}{N} \sum_{m \in [N]}\one_{ \|g(m)\|_{\R/\Z} \geq 1/4  } \geq \frac{1}{N}\sum_{m \in [N]} F(g(m)) \geq \frac{1}{8}\,.$$
	\end{proof}
	
	We are now prepared to make use of Theorem \ref{th:Green-Tao}; rather than proving Lemma \ref{lem:R_n} straight away, it will show that Lemma \ref{lem:R_n} holds for \emph{some} $\delta > 0$ rather than all $\delta > 0$:
	\begin{lemma}\label{lem:some-C}
		There exist constants $c,C > 0$ so that $$ \int_{R(C)} |\varphi(2\pi \bt)| \,d\mathbf{t}  = O(e^{-c \sqrt{n}})\,.$$ 
	\end{lemma}
	\begin{proof}
		Let $\mathbf{t} \in R(C)$ for $C \geq 1$ to be determined later. 
		Apply Theorem \ref{th:Green-Tao} and Lemma \ref{lem:rational-approx} with $\delta = 1/32$ to obtain constants $C_2, C_3$ so that either $(Q_\bt(x) \mod \Z )_{x \in [\sqrt{n}]}$ is $1/32$-equidistributed or there is a $k$ with $1 \leq k \leq C_2$ so that $$\| k Q_\bt  \|_{C^{\infty}[\sqrt{n}]} \leq C_3\,. $$ 
		
		We have three cases that we address separately:
		 
		\noindent \underline{Case 1: Approximately equidistributed}: 		If $(Q_\bt(x) \mod \Z )_{x \in [\sqrt{n}]}$ is $1/32$-equidistributed, then \begin{align*}
		\frac{| \{x \in [\sqrt{n}] : \| Q_\bt(x) \|_{\R/\Z} \geq 1/4  \} |}{\sqrt{n}} \geq \frac{1}{8}\
		\end{align*}
		by Lemma \ref{lem:approx-equid}.  Lemma \ref{lem:engine} shows $|\varphi(2\pi \bt)| \leq e^{- c\sqrt{n}}$ for some $c$ depending only on $K$.

		\underline{Case 2: Not equidistributed, but $k = 1$}: In this case, Lemma \ref{lem:rational-approx} implies that there is an integer-valued polynomial $q$ so that $\|t_j - q_j \|_{\R/\Z} \leq C_4 n^{-j/2}$ for all $j \in J$.  If we require $C \geq C_4$, then this would imply $\mathbf{t}\notin R(C)$, thus completing this case.
		
		\underline{Case 3: Not equidistributed, $k > 1$}: Suppose that $\| k Q_\bt \|_{C^\infty[\sqrt{n}]} \leq C_3$ for some $k > 1$ and \emph{not} $k = 1$.  Then there is an integer-valued polynomial $q$ so that $Q_\bt$ is close to $q/k$; further, since we know that $\| Q_\bt \|_{C^\infty[\sqrt{n}]} > C_3$, we have that $q$ is \emph{not} integer-valued.
		
		Since $q$ is integer-valued but $q/k$ is not, there exists some $n$ and integer $a$ so that $q(n) = a$ where $a \not\equiv 0 \mod k$.  Further, note that the polynomial $(q(z) - a)/k$ cannot be integer-valued since it has non-zero constant term.  Thus, there must be some $m$ and other value $b$ with $b \not\equiv a \mod k$ so that $q(m) = b$.  Write $q(z)/k = \frac{1}{B}\sum_{j \in J} a_j z^j$ where $B$ and each $a_j$ are integers.  Then note that for each integer $r$ with $r \equiv n \mod B$ we have $q(r)/k \equiv a/k \mod \Z$; similarly, if $r \equiv m \mod B$ then $q(r)/k \equiv b/k \mod \Z$.  Moreover, $\|b - a\|_{\R/(k\Z)} \geq \frac{1}{k}$.  Therefore for each $t_0$, we have that $\max \{ \|t_0 + b/k\|_{\R/\Z},\|t_0 + a/k\|_{\R/\Z}   \} \geq \frac{1}{2k} \geq \frac{1}{C_5}$.  Therefore, on some set of positive density, $Q_\bt$ is uniformly bounded away from $0$ on the torus $\R/\Z$.  Applying Lemma \ref{lem:engine} completes the proof.
	\end{proof}

	We are now ready to prove the lemma.
	\begin{proof}[Proof of Lemma \ref{lem:R_n}]
		In light of Lemma \ref{lem:some-C}, it is sufficient to find a constant $c = c(\delta,K)$ so that $$\int_{R(\delta) \setminus R(C)} |\varphi(2\pi\bt)| \,d\bt = O(e^{-c \sqrt{n}})\,.$$  Write $$R(\delta) \setminus R(C) = \bigcup_{q \in \QQ_J} (q + U(C)\setminus U(\delta))\,. $$
		
		Thus, it is sufficient to show the bound for each set in the above union.  Fix some $q \in \QQ_J$ and $\mathbf{t} \in q + U(C)\setminus U(\delta)$.  In the case where $0 \notin J$, simply set $t_0 = 0$ so that the following two cases make sense.
		
		\underline{Case 1: $\|t_0\|_{\R/\Z} \geq \delta$}.  Then for $m \leq  \delta \sqrt{n}/ (2 d C)$, we have \begin{align*}\left\| \sum_{j \in J} t_j m^j \right\|_{\R/\Z} &= \left\|t_0 + \sum_{j  \in J \setminus \{0\}} (t_j - q_j) m^j \right\|_{\R/\Z} \\
		&\geq \delta - \frac{\delta}{2} \\
		&= \delta/2\,. 
		\end{align*}
		Applying Lemma \ref{lem:engine} completes this case.
		
		\underline{Case 2: $\|t_0\|_{\R/\Z} < \delta$}.  The idea will be to look in the limiting setting and use a compactness argument.  Define the set of polynomials $$T := \left\{  \sum_{j \in J} c_j z^j : |c_i| \leq C \text{ for all }i \text{ and } |c_i| \geq \delta \text{ for some }i, |c_0| < \delta \right\}\,.$$
		
		Note that $T$ is compact; further, for $|z| \leq \delta /(2C)$ and all $p \in T$ we have $$|p(z)| < \delta + \delta = 2\delta\,. $$  
		
		Since $R(\delta)$ increases as $\delta$ decreases, we may assume without loss of generality that $\delta \leq 1/8$; in particular, this implies that $|p(z)| \leq 1/4$ for all $|z| \leq \delta/(2C)$. 
		
		For each $p \in T$, $p$ is not identically zero on $[0,\delta/(2C)]$ and so $|p|$ must attain a non-zero maximum $M(p)$.  Since $p \mapsto M(p)$ is a continuous function of the coefficients of $p$ and $T$ is compact, we must have that there is a value $M_0$ so that $M(p) \geq M_0$ for all $p \in S$.  Additionally, let $L(p)$ be the length of the maximum interval in $[0,\delta/(2C)]$ on which $|p| \geq M_0 / 2$.  Note that $L$ is non-zero on $T$ and continuous, and so we must have $|L(p)| \geq L_0$ for all $p \in T$.  
		
		For any $\mathbf{t}$ in the desired set, find $s_j \in (-1/2,1/2]$ so that $(t_j - q_j) \equiv s_j \mod \Z$.  Then the polynomial 
		$$p(z) := \sum_{j \in J} (s_j \cdot n^{j/2}) z^j $$ 
		lies in the set $T$.  Thus, there is an interval $I \subset [0,\delta/(2C)]$ of length at least $L_0$ on which $|p(z)| \geq M_0 /2$.  For any $x$ for which we have $x/N \in I$ compute \begin{align*}
		\left\| \sum_{j \in J} t_j x^j \right\|_{\R/\Z} &= \left\| \sum_{j \in J} s_j x^j \right\|_{\R/\Z} \\
		&=  \left\| \sum_{j \in J} (s_j n^{j/2}) (x/\sqrt{n})^j \right\|_{\R/\Z} \\
		&= \left|\sum_{j \in J} (s_j n^{j/2}) (x/\sqrt{n})^j \right| \\
		&\geq M_0 / 2
		\end{align*}
		where in the last equality we used the fact that $|p(z)| \leq 1/4$ for $|z| \leq \delta/(2C)$.  Since $I$ is an interval, we have that the number of $x$ for which $x/\sqrt{n} \in I$ is at least $\sqrt{n}\cdot |I| - 1$.  Since $|I|$ is bounded below by $L_0$, the proof is complete after applying Lemma \ref{lem:engine}.
	\end{proof}
	
	\subsection{Combining the integrals} \label{sec:combined}
	The primary goal of this section is to show the following lemma.
	\begin{lemma}\label{lem:combined}
		For each $\delta > 0$ there is a constant $c = c(\eps,\delta,d) > 0$ so that $$\sup_{\ba}\left|\P(\bX = \ba) - \one_{\ba \in \NT}|\QQ_J| \int_{U} \varphi(2 \pi \bt)e^{-i\langle 2\pi\bt,\ba\rangle}\,d\bt\right| = O(e^{-c\sqrt{n}})\,. $$
	\end{lemma} 
	
	A first step is a simple lemma that changes coordinates to combine the integrals from the previous section.  For a polynomial $q \in \QQ_J$, we write $\mathbf{q} \in \R^{J}$ for the vector of coefficients of $q$.  
	\begin{lemma}\label{lem:combine-integrals}
		For $n$ sufficiently large we have 
		$$\int_{\cup_{q \in \QQ_J} (q + U) } \varphi(2\pi \bt)e^{-i\langle 2\pi\bt,\ba\rangle} \,d\mathbf{t} = \left(\sum_{q \in \QQ_J} e^{-i\langle 2\pi\bq, \ba\rangle} \right)\int_{U}  \varphi(2\pi\bt)e^{-i\langle 2\pi \bt,\ba\rangle }  \,d\mathbf{t}\,.  $$
	\end{lemma}
	\begin{proof}
		For $n$ sufficiently large, we have that the union is in fact disjoint.  This means that we may first write \begin{equation*}
		\int_{\cup_{q \in \QQ_J} (q + U) } \varphi(2\pi \bt)e^{-i\langle 2\pi \bt,\ba\rangle} \,d\mathbf{t} = \sum_{q \in \QQ_J} \int_{q + U} \varphi(2\pi\bt )e^{-i\langle 2\pi \bt,\ba\rangle} \,d\mathbf{t}\,.
		\end{equation*}
		
		For each $q \in \QQ_J$, we have \begin{align*}
		\int_{q + U} \varphi(2\pi\bt)e^{-i\langle 2\pi \bt,\ba \rangle} \,d\mathbf{t} 
		&=  \int_U \varphi(2\pi(\bt + \bq) )e^{-i\langle 2\pi (\bt + \bq),\ba\rangle}\,d\mathbf{t} \\
		&= e^{-i\langle 2\pi \bq, \ba \rangle}\int_U \prod_{x \geq 1} \varphi_x(2\pi i \langle \bt + \bq,\bx\rangle)  e^{-i\langle 2\pi \bt,\ba\rangle } \,d\bt \\
		&= e^{-i\langle 2\pi \bq, \ba \rangle} \int_U \prod_{x \geq 1} \varphi_x(2\pi i \langle \bt,\bx\rangle)  e^{-i\langle 2\pi \bt,\ba\rangle } \,d\bt  \\
		&= e^{-i\langle 2\pi \bq, \ba \rangle} \int_U \varphi(2\pi \bt) e^{-i\langle 2\pi \bt,\ba\rangle } \,d\bt\,.
		\end{align*}
	\end{proof}

	\begin{lemma}\label{lem:NT}
		For each vector of integers $ \mathbf{m} = (m_j)_{j \in J}$ if $\mathbf{m} \in \NT$ then we have \begin{equation}
		\sum_{q \in \QQ_J} e^{-2\pi i\langle \bq,\mathbf{m}\rangle} = |\QQ_J| \,.
		\end{equation} 
	\end{lemma}
	\begin{proof}
		Let $q \in \QQ_J$.  Then since $\mathbf{m} \in \NT$ we have 
		$$\sum_{j\in J} q_j m_j \equiv 0 \mod \Z$$
		and so
		$$\sum_{q \in \QQ_J} e^{-2\pi i \langle \mathbf{q},\mathbf{m} \rangle} = \sum_{q \in \QQ_J} 1 = |\QQ_J|\,.$$
	\end{proof}

	\begin{proof}[Proof of Lemma \ref{lem:combined}]
		Note that if $\ba \notin \NT$ then $\P(\bX = \ba) = 0$.  We may thus write \begin{align*}
		\P(\bX = \ba) &= \one_{\ba \in \NT} \int_{(-1/2,1/2]^{J}} \varphi(2\pi \bt) e^{-i \langle 2\pi \bt, \ba \rangle}\,d\bt \\
		&= \one_{\ba \in \NT} \int_{\cup q \in \QQ_J (q + U)} \varphi(2\pi \bt) e^{-i \langle 2\pi \bt, \ba \rangle}\,d\bt + O(e^{-c\sqrt{n}}) \\
		&= \one_{\ba \in \NT} \left(\sum_{q \in \QQ_J} e^{-i\langle 2\pi \bq,\ba\rangle} \right)\int_{ U} \varphi(2\pi \bt) e^{-i \langle 2\pi \bt, \ba \rangle}\,d\bt + O(e^{-c\sqrt{n}}) \\
		&= \one_{\ba \in \NT}  |\QQ_J|\int_{ U} \varphi(2\pi \bt) e^{-i \langle 2\pi \bt, \ba \rangle}\,d\bt + O(e^{-c\sqrt{n}})
		\end{align*}
		by applying \eqref{eq:prob-inversion} and then Lemmas \ref{lem:R_n}, \ref{lem:combine-integrals} and \ref{lem:NT} in succession.
	\end{proof}

	\subsection{Approximating the integral over $U$}\label{subsecApproxOverU}
	With Lemma \ref{lem:combined} established, we turn to the integral over $U$.  First we show that an expression in the statement of Theorem \ref{th:LCLT} can be written as a Gaussian integral.  
	\begin{lemma} \label{lem:eval-Gaussian-integral} For any vectors $\bt,\ba \in \R^{d}$ and positive definite $d\times d$ matrix $S$ we have \begin{equation}\label{eq:Gaussian-integral}
		\frac{1}{(2\pi)^{d}}\int_{\R^{d}} \exp\left(-i \langle \bt, \ba - \E \bX\rangle  - \frac{1}{2} \bt^T S \bt \right)\,d\bt = \frac{\exp\left(-\frac{1}{2}(\mathbf{a} - \E \bX)^T S^{-1} (\mathbf{a} - \E \bX)\right)}{\sqrt{(2\pi)^{d}\det(S) } }\,.
		\end{equation}
	\end{lemma}
	\begin{proof}
		Since $S$ is positive definite, there is an invertible matrix $M$ so that $S = M^T M$.  We may thus write \begin{align*}
		-i \bt^T (\ba - \E \bX) - \frac{1}{2} \bt^T S \bt 	&= - \frac{1}{2}(M \bt + i (M^{-1})^T(\ba - \E \bX ))^T(M \bt + i (M^{-1})^T(\ba - \E \bX )) \\
		&\qquad - \frac{1}{2}(\ba - \E \bX)^T (S)^{-1} (\ba - \E \bX)\,.
		\end{align*}
		
		Setting $\mathbf{s} = M \mathbf{t} + i(M^{-1})^T (\mathbf{a} - \E \bX)$ and using Cauchy's integral theorem shows \begin{align*}
		\int_{\R^{d}} &\exp\left(-i \langle \bt, \ba - \E \bX\rangle  - \frac{1}{2} \bt^T S \bt \right)\,d\bt\\
		&\quad= \frac{\exp\left(-\frac{1}{2}(\mathbf{a} - \E \bX)^T (S)^{-1} (\mathbf{a} - \E \bX)\right)}{\det(M)}\int_{\R^{d}}\exp(-\frac{1}{2}\| \mathbf{s}\|^2 )\,d\mathbf{s} \,.
		\end{align*}
		
		Evaluating the Gaussian integral as $(2\pi)^{d/2}$ and recalling $\det(M) = \det(S)^{1/2}$ completes the proof of \eqref{eq:Gaussian-integral}.
	\end{proof}
	
	Recall that for a given $\delta > 0$, we have defined $U = U(\delta) = \{\bt : |t_j| \leq \delta n^{-j/2}  \}$.  With Lemmas \ref{lem:combined} and \ref{lem:eval-Gaussian-integral} in tow, it is sufficient to show that for some $\delta > 0$, we have \begin{equation}\label{eq:need}
	\int_{\R^{J}}\left|\varphi(\bt)\cdot\one_{\{\bt \in U  \}} - \exp\left(  i\langle \bt,\E \bX \rangle - \frac{1}{2} \bt^T S \bt \right)  \right|\,d\bt = o(n^{- \sum_{j \in J}(j/2+1/4)})\,.
	\end{equation}
	
	This bound will follow from showing the following estimates:
	\begin{align}\label{eq:u-comp}
	\int_{U^c} \exp\left(-\frac{1}{2}\bt^T S\bt \right) \,d\bt &= O(e^{-c \sqrt{n}}) \\
	\label{eq:U-minus-V}n^{- \sum_{j \in J}(j/2+1/4)}\int_{U \setminus V} \left|\varphi(\bt) - \exp\left(  i\langle \bt,\E \bX \rangle - \frac{1}{2} \bt^T S \bt \right)\right| \,d\bt &= O(e^{-c' (\log n)^2}  ) \\
	\label{eq:V} n^{- \sum_{j \in J}(j/2+1/4)}\int_{V} \left|\varphi(\bt) - \exp\left(  i\langle \bt,\E \bX \rangle - \frac{1}{2} \bt^T S \bt \right)\right| \,d\bt &= O\left(\frac{(\log n)^3}{n^{1/8}} \right)
	\end{align}
	where we have defined $V := \{\bt : |t_j| \leq (\log n) n^{-j/2-1/4} \}$.
	
	To show \eqref{eq:u-comp} along with the bound on $\det(S)$ stated in Theorem \ref{th:LCLT}, we demonstrate upper and lower bounds on $S$ when viewed as a quadratic form.  In what follows, we set $\sigma_k^2 = \Var(Y_k) = (1 - p_k)/p_k^2$ and $S_k = \Cov( (Y_k k^j)_{j \in J})$.

	\begin{lemma} \label{lem:Sigma-scale}
		In the context of Theorem \ref{th:LCLT}, define the matrix $\widetilde{S}$ via $\widetilde{S}_{i,j} = S_{i,j}/n^{(i+j+1)/2}$ for all $i,j \in J$.  Then there is a constant $C = C(K) > 0$ so that \begin{equation}\label{eq:sigma-bound}
		C^{-1} \| \ba \|^2 \leq  \ba^T \widetilde{S} \ba \leq C \| \ba \|^2
		\end{equation}
		for all $\ba \in \R^{J}$.  In particular, $\det(S) = \Theta( n^{ \sum_{j \in J}(j + 1/2)}  )$.
	\end{lemma}
	\begin{proof}
		By rescaling, assume without loss of generality that $\| \ba \| = 1$.
		Compute \begin{equation}\label{eq:quadratic-expansion}
		\ba^T \widetilde{S} \ba = \frac{1}{\sqrt{n}}\sum_{k = 1}^\infty \sigma_k^2 \left(\sum_{j \in J} a_j (k/\sqrt{n})^j  \right)^2
		\end{equation}
		where $\sigma_k^2 = \Var(Y_k)$.  If we define $$f_{\widetilde{\bbeta}}(t) = \frac{\exp\left(- \sum_{j \in J} \widetilde{\bbeta}_j t^j \right)}{\left(1- \exp\left(- \sum_{j \in J} \widetilde{\bbeta}_j t^j \right) \right)^2}$$
		then $\sigma_k^2 = f_{\widetilde{\bbeta}}(k/\sqrt{n})$, where we have written $\widetilde{\bbeta}= (\bbh_j n^{j/2})_{j \in J}$.  We then have 
		\begin{equation}\label{eq:quadratic-int}
		\ba^T \widetilde{\Sigma} \ba = \frac{1}{\sqrt{n}} \sum_{k \geq 1} f_{\bbt}(k/\sqrt{n}) \left( \sum_{j\in J} a_j (k/\sqrt{n})^j \right)^2 = \int_0^\infty f_{\bbt}(x) \left(\sum_{j \in J} a_j x^j \right)^2\,dx  + E
		\end{equation}
		where $$E \leq \frac{1}{\sqrt{n}}(g_{\bbt}(1/\sqrt{n}) + g_{\bbt}'(1/\sqrt{n}) + \int_0^{1/\sqrt{n}} g_{\bbt}(x)\,dx + \frac{1}{n}\int_{0}^\infty |g_{\bbt}''(x)| \,dx $$
		where $g_{\bbt}(x) = f_{\bbt}(x)(\sum_{j \in J} a_j x^j)^2$.  Bounding each term in a straightforward manner gives $|E| =  O(1/\sqrt{n})$ where the error is uniform since $\bbt$ varies in a compact set. 	
		
		Since $\mathbf{a}$ and $\bbt$ vary over compact sets, the integral on the right-hand-side of \eqref{eq:quadratic-int} is bounded above and below away from $0$; thus, for $n$ sufficiently large, \eqref{eq:quadratic-int} demonstrates uniform bounds above and below on $\ba^T \widetilde{\Sigma} \ba$.  For each remaining small $n$, the sum $$\frac{1}{\sqrt{n}} \sum_{k \geq 1} f_{\blt}(k/\sqrt{n}) \left( \sum_{j \in J} a_j (k/\sqrt{n})^j \right)^2$$ may be uniformly bounded above and below by compactness and continuity, thereby completing the proof. 
	\end{proof}

	The bound \eqref{eq:u-comp} now follows easily from Lemma \eqref{lem:Sigma-scale}: for a given $\bt$ define $\mathbf{s}$ via $s_j = t_j n^{j/2 + 1/4}$; then 
	\begin{align*}
	\int_{U^c} \exp\left(-\frac{1}{2}\bt^T S\bt \right) \,d\bt &\leq \int_{\R^{J} \setminus [-\delta n^{1/4},\delta n^{1/4}]^{J}} \exp\left(-c \| \mathbf{s} \|^2 \right) \,d\mathbf{s} \\
	&= O(e^{-c' \sqrt{n}})
	\end{align*}
	for some constant $c' > 0$.  The next two sections show \eqref{eq:U-minus-V} and \eqref{eq:V}.

	\subsection{Establishing \eqref{eq:U-minus-V}}\label{subsecU-minus-V}

	To show \eqref{eq:U-minus-V}, we will show that for a geometric variable $Y$ with mean bounded above we can compare the characteristic function of $Y$ to that of a corresponding Gaussian with an error depending $\Var(Y)$  (this is Lemma \ref{lem:cumulant-bound} below).  In the problem at hand, our geometric variables actually have unbounded means; our first step is to show that while the means can be unbounded, the bulk of the contribution to the variance $S$ comes from variables of bounded mean.   In this direction, we alter the proof of Lemma \ref{lem:Sigma-scale} to show that the contribution of the variance from the first $\eps \sqrt{n}$ terms can be made to be less than half provided $\eps$ is small enough:
	
	\begin{lemma}\label{lem:truncate-var}
		There exists an $\eps > 0$ so that for all $\bt \in U$ we have $$ \bt^T S \bt \geq \frac{1}{2} \bt^T S^{(\eps)} \bt$$
		where  $$S^{(\eps)} := \sum_{k \geq  \sqrt{n} \eps}S_k\,.$$
	\end{lemma}
	\begin{proof}
		Define $\mathbf{s}$ via $s_j = t_j n^{j/2+1/4}$.  Then $$\bt^T S \bt = \mathbf{s}^T \widetilde{S} \mathbf{s}\,. $$
		
		We then want to show $$\mathbf{s}^T \widetilde{S} \mathbf{s} \geq \frac{1}{2} \mathbf{s}^T \widetilde{S}^{(\eps)} \mathbf{s}$$ where $\widetilde{S}^{(\eps)}$ is defined in the analogous way.  Since both sides are homogeneous in $\mathbf{s}$, assume without loss of generality that $\mathbf{s}$ is a unit vector. 
		
		The proof of Lemma \ref{lem:Sigma-scale} shows \begin{align*}
		\mathbf{s}^T \widetilde{S}^{(\eps)} \mathbf{s} &= \int_\eps^\infty \left(\sum_{j \in J} s_j x^j \right)^2 \sigma_x^2\,dx + O(1/\sqrt{n}) 
		\end{align*}
		where the error may be taken uniformly over $\eps$.  By compactness we have that for each $\eps > 0$ that $$ \max_{\|\mathbf{s}\|=1} \int_0^\eps \left(\sum_{j \in J} s_j x^j \right)^2 \sigma_x^2 \,dx$$
		is bounded above and tends to zero as $\eps \to 0^+$.  Similarly, we also have $$\min_{\|\mathbf{s}\|=1} \int_0^\infty \left(\sum_{j \in J} s_j x^j \right)^2 \sigma_x^2 \,dx =: c > 0\,.$$
		
		Choosing $\eps$ small enough then gives the desired bound.
	\end{proof}

	Seeking to show our characteristic function of $\bX$ is close the characteristic function of the corresponding Gaussian, we need a tail bound on the cumulant generating function $\log \E e^{it Y}$ for geometric random variables.  This will allow us to approximate the characteristic function for each $Y$ with that of a Gaussian with mean $\E Y$ and variance $\Var(Y)$. 	
	
	\begin{lemma}\label{lem:cumulant-bound}
		Fix $\eps > 0$.  Let $Y$ be a geometric random variable for some $p \in [\eps,1)$.  Then there are constants $C, c > 0$ so that for all $|t| \leq c$, we have $$\left| \log \E e^{i t Y} - \left(i t \E Y - \frac{t^2}{2} \Var(Y)\right)\right| \leq C \Var(Y) |t|^3\,.$$
	\end{lemma}
	\begin{proof}
		For convenience, write $\mu = \E Y$ and $\sigma^2 = \Var(Y)$ and set $\widehat{Y} = Y - \mu$.  Then \begin{align*}
		\left|\E e^{i t (Y - \mu)} - (1 - \sigma^2 t^2/2) \right| \leq |t|^3 \E \left|Y - \mu\right|^3 \,.
		\end{align*}
		We bound $$\E|Y - \mu|^3 \leq 8 (\E|Y^3| + \mu^3) = 8\left(\frac{(1 - p)(p^2 - 6p + 6)}{p^3} + \frac{(1-p)^3}{p^3} \right) \leq C_\eps \frac{1 - p}{p^2} = C_\eps \sigma^2\,.$$
		Note $\E e^{itY} = \frac{p}{1 + (1 - p)e^{it}}$ and so $|\arg \E e^{itY}| = |\arg(1 + (1 - p)e^{it})|$ which is uniformly bounded since $1 - p$ is uniformly bounded away from $1$.  Write \begin{align*}
		\left| \log \E e^{i t (Y - \mu)} + \frac{t^2}{2} \sigma^2\right| &= \left| \log \left(\E e^{i t (Y - \mu)}  e^{\sigma^2 t^2 / 2}\right)\right| \\
		&= \left|\log\left(1 - \left(\frac{e^{-t^2 \sigma^2 / 2} - \E e^{it(Y - \mu)}}{e^{-t^2 \sigma^2/2}} \right)\right)\right|  \\
		&\leq C \left|\frac{ \E e^{i t(Y - \mu)}  - e^{-t^2 \sigma^2 / 2}}{e^{-t^2 \sigma^2 / 2}} \right| \\
		&\leq C' \left| \E e^{i t(Y - \mu)} - (1 - t^2 \sigma^2 / 2)\right| + O(|t|^4 \sigma^4) \\
		&\leq C'' |t|^3 \sigma^2 \, ,
		\end{align*}
		where each inequality uses the fact that $|t|$ and $\sigma$ are uniformly bounded above.
	\end{proof}

	In a similar vein to Lemma \ref{lem:Sigma-scale}, we need the following simple bound whose proof is omitted since it is essentially the same as that of Lemma \ref{lem:Sigma-scale}. \begin{lemma}\label{lem:sigma-int}
		$$\frac{1}{\sqrt{n}} \sum_{k \geq 1} \sigma_k^2 \left(\sum_{j \in J} (k/\sqrt{n})^j \right)^3 = O(1)\,.$$
	\end{lemma}

	Now define $T_k:=\{ \bt : k \leq \max_{j} n^{j/2 + 1/4} |t_j| \leq k+1 \}$ and note that $U \setminus V = \bigcup_{ k = 1}^{\delta n^{1/4}} T_k$.  
	\begin{lemma}\label{lem:Sk}
		There exist constants $c,C$ so that for $\bt \in T_k$ we have $$\bt^T S \bt \geq c k^2$$ and $$\sum_{x \geq \eps \sqrt{n}}\left|\log \varphi_r(\bt) - \left(i \langle \bt,\mu_r \rangle - \frac{1}{2}\bt^T S \bt \right) \right| \leq C \delta k^2\,.$$
	\end{lemma}
	\begin{proof}
		The first statement follows from \eqref{eq:sigma-bound}.  For the second, note that for $\bt \in T_k$ we have \begin{align*}
		\sum_{x \geq \eps \sqrt{n}} \left|\log \varphi_x(\bt) - \left(i \langle \bt,\E Y_x \rangle - \frac{1}{2}\bt^T S_x \bt \right) \right| &\leq C \sum_{x \geq \eps \sqrt{n}} \sigma_x^2 |Q_{\bt}(x)|^3 \\
		&\leq C \sum_{x \geq 1} n^{-3/4}\sigma_x^2 \left(\sum_{j \in J} (x/\sqrt{n})^j (k+1) \right)^3 \\
		&\leq C' k^3 n^{-3/4} \sum_{x \geq 1} \sigma_x^2 \left(\sum_{j \in J} (x/\sqrt{n})^j  \right)^3 \\
		&\leq C' \delta k^2 n^{-1/2} \sum_{x \geq 1} \sigma_x^2 \left(\sum_{j\in J} (x/\sqrt{n})^j  \right)^3 \\
		&\leq C'' \delta k^2 \,,
		\end{align*}
		where the first bound is by Lemma \ref{lem:cumulant-bound} 
		and the last bound is via Lemma \ref{lem:sigma-int}.
		
	\end{proof}
	
	As an immediate result, we see that if $\bt \in T_k \cap U$ for large $k$ then $|\varphi(\bt)|$ is quite small.
	\begin{cor} For $\delta$ sufficiently small, there is a constant $c > 0$ so that for $\bt \in T_k \cap U$ we have $$\left|\varphi(\bt)\right| \leq  e^{-c k^2}\,.$$
	\end{cor}
	\begin{proof}
		Bound	
		\begin{align*}
		|\varphi(\bt)| &\leq \prod_{x \geq \eps n^{1/4}} |\varphi_x(\bt)| \\
		&\leq \exp\left(- (c/4 - C\delta)k^2  \right) \\
		&\leq \exp(-c' k^2)\,,
		\end{align*}
		where we have chosen $\delta$ sufficiently small and used Lemma \ref{lem:truncate-var}.
	\end{proof}

	Since the measure of the set $\{\mathbf{b}: k \leq |b_j| \leq k+1 \}$ is equal to some polynomial $p(k)$, the measure of $T_k$ is $|T_k| = p(k) n^{-\sum_{j \in J}(j/2+1/4)}$.
	
	With these preliminaries in place, we are ready to tackle \eqref{eq:U-minus-V}:
	
	\begin{proof}[Proof of \eqref{eq:U-minus-V}]
		
		To begin with, bound \begin{equation*}
		\int_{U\setminus V} \left|\varphi(\bt) - \exp(i\langle \bt,\E \bX \rangle-\frac{1}{2}\bt^T S\bt) \right|\,d\bt \leq \sum_{k = \log n}^{\lceil \delta n^{1/4} \rceil} |T_k| \max_{\bt \in T_k}  \left|\varphi(\bt) - \exp(i\langle \bt,\E \bX\rangle-\frac{1}{2}\bt^T S\bt) \right|\,,
		\end{equation*}
		then compute \begin{align*}
		\sum_{\log n \leq k \leq \delta n^{1/4}} |T_k| \max_{\bt \in T_k} \left|\varphi(\bt) - \exp(-\frac{1}{2}\bt^T S\bt) \right| &\leq n^{-\sum_{j \in J}(j/2+1/4) } \sum_{k \geq \log n} p(k) \exp(-c k^2) \\
		&= O\left(e^{-c' (\log(n))^2 }\right) \,.
		\end{align*}
	\end{proof}
		
	\subsection{Establishing \eqref{eq:V}}\label{subsecV}  Our proof of \eqref{eq:V} can be viewed as an adaptation of the classical proof of the Lindeberg-Feller central limit theorem with an explicit error bound; see, for instance, \cite[Chapter 3.4]{durrett-book} for a similar proof and discussion of the classic theorem.

	We first bound contribution of the tail to the second moment of a geometric variable.  
	\begin{lemma}\label{lem:geo-tail}
		Let $X$ be a geometric random variable with parameter $p$. Then for each $C > 0$ we have $$ \E\left(|X - \E X|^2 \one_{\{ |X - \E X| \geq C\}} \right) \leq \frac{3 \Var(X)}{C p}\,.$$
	\end{lemma}
	\begin{proof}
		By the Cauchy-Schwarz inequality we have \begin{align*}
		\E\left(|X - \E X|^2 \one_{\{ |X - \E X| \geq C\}} \right) \leq \sqrt{\E\left(|X - \E X|^4 \right) \P(|X - \E X| \geq C)}\,.
		\end{align*}
		Compute \begin{align*}
		\E|X - \E X|^4 &= \sum_{k \geq 0} \left(k - \frac{1-p}{p} \right)^4 p \cdot (1 - p)^k = \frac{(1 - p)(9 - 9p + p^2)}{p^4} \\
		&\leq \frac{9 (1 - p)}{p^4} = \frac{9 \Var(X)}{p^2}\,.
		\end{align*}
		By Chebyshev's inequality, bound $$ \P(|X - \E X| \geq C) \leq \frac{\Var(X)}{C^2}\,.$$
		Putting the three equations together completes the proof.
	\end{proof}
	
	We are now ready to show \eqref{eq:V}.
	\begin{proof}[Proof of \eqref{eq:V}]
		Fix $\bt \in \R^J$ with $\| \bt \| \leq 1$ and let $\eps > 0$ to be chosen later.  Consider the random variables $$Z_k := Z_k(\bt) := (Y_k - \E Y_k) \sum_{j \in J} k^j t_j n^{-j/2-1/4}\,.$$
		
		Then for $|\theta| \leq \log 
		n$ note that $$\varphi_{\widehat{\bX}}(\theta \bt) = \prod_{k \geq 1} \E e^{i\theta Z_k}$$
		and so
		\begin{align*} \left|\varphi_{\widehat{\bX}}(\theta \bt) - \exp\left(- \theta^2\bt^T {S} \bt/2 \right) \right| &\leq  \left|\exp\left(- \theta^2\bt^T {S} \bt/2 \right) - \prod_{k \geq 1} (1 - \theta^2 \E Z_k^2 /2) \right|\\
		&\qquad+ \sum_{k \geq 1} \left|\E[e^{i \theta Z_k}] - (1 - \theta^2 \E Z_k^2 /2)\right|\,.
		\end{align*}
		
		Note that \begin{align*}
		\prod_{k \geq 1} (1 - \theta^2 \E Z_k^2 /2) &= \exp\left(\sum_{k \geq 1} \log(1 - \theta^2 \E Z_k^2 / 2) \right) \\
		&= \exp\left(-\theta^2 \bt {S} \bt/2 + \sum_{k \geq 1} O(\theta^4 (\E Z_k^2)^2)\right) \\
		&= \exp\left(-\theta^2 \bt {S} \bt/2 + O(1/\sqrt{n})\right)
		\end{align*}
		where the third line is obtained by comparing a sum to an integral as in the proof of Lemma \ref{lem:Sigma-scale}.  It is thus sufficient to show \begin{equation}
		\max_{\|\bt\| \leq 1, |\theta| \leq \log n}\sum_{k \geq 1} \left|\E[e^{i \theta Z_k}] - (1 - \theta^2 \E Z_k^2 /2)\right| = O\left(\frac{(\log n)^3}{n^{1/8}} \right)\,.
		\end{equation}
		Working towards this, recall that for any mean-zero variable $X$ we have \begin{equation}\label{eq:333}\left|\E[e^{i t X}] - (1 - t^2 \E X^2 /2)\right| \leq \E\left[ \min\{|tX|^3,2|t X|^2  \}\right]\,.
		\end{equation}
		See, for instance, \cite[Equation (3.3.3)]{durrett-book}.  Applying this bound to $Z_k$ gives

		\begin{align*}
		\left|\E[e^{i \theta Z_k}] - (1 - \theta^2 \E Z_k^2 /2)\right| &\leq \E\left[ \min\{|\theta Z_k|^3,2|\theta Z_k|^2   \}\right] \\
		&\leq \E\left[ |\theta Z_k|^3 \one\{  |Z_k| \leq \eps\} \right] + \E\left[2|\theta Z_k|^2   \one\{ |Z_k| > \eps \} \right] \\
		&\leq |\theta|^3 \eps \E[|Z_k|^2 \one\{ |Z_k| \leq \eps \} ] + 2 \theta^2 \E[|Z_k|^2 \one\{ |Z_k| > \eps \} ] \\
		&\leq  |\theta|^3 \eps \Var(Z_k) + 2 \theta^2 \E[|Z_k|^2 \one\{ |Z_k| > \eps \} ]   \,.
		\end{align*}

		Using \eqref{eq:333} and Lemma \ref{lem:geo-tail}, bound \begin{align*}
		\sum_{k \geq 1} \left|\E[e^{i \theta Z_k}] - (1 - \theta^2 \E Z_k^2 /2)\right| &\leq \eps |\theta|^3 \bt^T \widetilde{S} \bt  + 2 \theta^2 \sum_{k \geq 1} \E[|Z_k|^2 \one\{|Z_k| > \eps \}]\\
		&\leq \eps |\theta|^3 \bt^T \widetilde{S} \bt + 6 \frac{\theta^2}{n^{1/4} \eps}\cdot \left( \frac{1}{\sqrt{n}} \sum_{k \geq 1} \left(\sum_{j \in J} k^j t_j n^{-j/2} \right)^3 \frac{\sigma_k^2}{p_k} \right)\\
		&= O\left(\eps |\theta|^3 + \frac{\theta^2}{n^{1/4} \eps} \right)\,,
		\end{align*}
		where in the last line we compared a sum to an integral to obtain this error term. 		
		Choosing $\eps = n^{-1/8}$ and bounding $|\theta| \leq \log n$ completes the proof.
	\end{proof}

	\begin{proof}[Proof of Theorem \ref{th:LCLT}]
		Chaining together Lemmas \ref{lem:combined} and \ref{lem:eval-Gaussian-integral} with \eqref{eq:need} completes the proof.
	\end{proof}
	
	Theorem~\ref{thmIntegralNonDistinct} follows immediately from Corollary~\ref{corPartitionNumberFormula}, Lemma~\ref{lemEntropyApprox}, and Lemma~\ref{lemCLTapprox}.

\section{Limit Shape} \label{secLimitShape}

In order to show convergence in distribution of $\{\phi_{\lambda,n}\}$ to $\{\phi_{\infty}\}$ we need to show that for each $0 < t_1 < t_2 < \infty$ and $\eps > 0$ we have \begin{equation}\label{eq:limit-need}
\lim_{n \to \infty} \P_{\lambda \sim \mathcal{P}(\bN(\balpha,n))}\left( \max_{t \in [t_1,t_2]} \left| \phi_{\lambda,n} - \phi_{\infty} \right| \geq \eps  \right) \to 0
\end{equation}
where the probability takes $\lambda$ uniformly from $\mathcal{P}(\bN(\balpha,n))$.  

The idea will be to show that if $\lambda$ is chosen according to the maximum entropy measure $\mu_n$, then the convergence in $\eqref{eq:limit-need}$ is exponentially small in $\sqrt{n}$; using Lemma \ref{lemCLTapprox} together with \eqref{eqpnMu} will complete the proof.

Adopting the notation of Theorem \ref{th:LCLT}, let $\{X_k\}_{k \geq 1}$ be independent geometric random variables where $X_k$ has mean $(\exp(\sum_{j \in J} \bbh k^j) - 1)^{-1}$ and recall that this is the number of parts of size $k$ of a partition $\lambda$ chosen according to the maximum entropy measure $\mu_n$.  

The core of the proof is to show that the heuristic \eqref{eq:LLN} holds on the scale of $\sqrt{n}$ holds: 

\begin{lemma}\label{lem:geo-concentrate}
	Let $K \subset (0,\infty)$ be a compact set. Then for each $t \in K$ and $\eps > 0$ there exists a constant $C_{K,\eps}$ so that $$\P\left(  \left|\sum_{k \geq t\sqrt{n}} X_k - \int_t^\infty \frac{1}{\exp\left(\sum_{j \in J} \bbeta_j x^j\right)-1}\,dx  \right| \geq \eps \sqrt{n} \right) \leq e^{-C_{K,\eps} \sqrt{n}}\,.$$
\end{lemma}
\begin{proof}
	Note that for $n$ sufficiently large (uniformly in $K$) we have \begin{align*} &\left|\sum_{k \geq t \sqrt{n}} \E X_k - \int_{t}^\infty \frac{1}{\exp\left(\sum_{j \in J} \bbeta_j x^j\right)-1}\,dx \right| \\
	&\qquad= \left|\sum_{k \geq t \sqrt{n}}\frac{1}{\exp\left(\sum_{j \in J} \bbh_j k^j\right)-1} - \int_t^\infty \frac{1}{\exp\left(\sum_{j \in J} \bbeta_j x^j\right)-1}\,dx \right| \\
	&\qquad \leq \eps \sqrt{n}/2
	\end{align*}
	
	and so it is sufficient to show \begin{equation}
	\P\left(\left|\sum_{k \geq t\sqrt{n}} (X_k - \E X_k) \right| \geq \eps\sqrt{n}/2 \right) \leq \exp(-C_{K,\eps}\sqrt{n})\,.
	\end{equation}
	This will follow from a standard Chernoff bound argument.  By Lemma \ref{lem:cumulant-bound}, for each $k \geq t \sqrt{n}$ there are constants $C,c > 0$ (depending on $K$) so that for $|\theta| \leq c$ we have $$\E \exp(\theta (X_k - \E X_k)) \leq \exp(C \theta^2 \Var(X_k))\,.$$ 
	
	This implies for $|\theta| \leq c$ we have $$\E \exp\left(\theta \sum_{k \geq t \sqrt{n}}(X_k - \E X_k) \right) \leq \exp\left(C \theta^2 \sum_{k \geq t\sqrt{n}} \Var(X_k)\right) \leq \exp(C' \theta^2 \sqrt{n})\,.$$

	Applying this bound along with Markov's inequality for $\theta \in (0,c]$ to be chosen small enough with respect to $\eps$ \begin{align*}
	\P\left(\sum_{k \geq t\sqrt{n}} (X_k - \E X_k)  \geq \eps\sqrt{n}/2 \right) &= \P\left(\exp\left(\theta \sum_{k \geq t\sqrt{n}}(X_k - \E X_k)\right) \geq \exp(\eps \theta \sqrt{n}/2)  \right) \\
	&\leq \exp(C' \theta^2 \sqrt{n} - \eps \theta \sqrt{n}/2) \\
	&\leq \exp(-C_\eps \sqrt{n})\,.
	\end{align*}
	The corresponding lower bound follows by an identical argument.
\end{proof}
	We now show that \eqref{eq:limit-need} holds for $\lambda$ chosen according to $\mu_n$:
	
\begin{cor}
	For $0 < t_1 < t_2 < \infty$ and $\eps > 0$, there is a constant $c = c(t_1,t_2,\eps) > 0$ so that 
	$$\P_{\lambda \sim \mu_n} \left( \max_{t \in [t_1, t_2]} |\phi_{\lambda,n}(t) - \phi_\infty(t)| \geq \eps  \right) \leq e^{-c \sqrt{n}}\,.$$
\end{cor}	
\begin{proof}
	For each $t \in [t_1,t_2]$ so that $t\sqrt{n} \in \Z$, apply Lemma \ref{lem:geo-concentrate} to see $$|\phi_{\lambda,n}(t) - \phi_\infty(t)| \geq \eps \leq e^{-c \sqrt{n}}\,.$$
	Since there are $O(\sqrt{n})$ many such $t$, we may union bound $$\P_{\lambda \sim \mu_n} \left( \max_{t \in [t_1, t_2]} |\phi_{\lambda,n}(t) - \phi_\infty(t)| \geq \eps  \right) \leq C \sqrt{n} e^{-c \sqrt{n}} \leq e^{-c'\sqrt{n}}\,.$$
	
	\iffalse	
	Fix $\delta > 0$ to be chosen later; apply Lemma \ref{lem:geo-concentrate} for all $t$ with $t\sqrt{n} \in \Z$ and $t \in [t_1,t_2]$ to see $$\P\left( \max_{t \in [a,b]} \left|\sum_{k \geq t\sqrt{n}} X_k - \int_t^\infty \frac{1}{\exp\left(\sum_{j \in J} \bbeta_j x^j\right)-1}\,dx  \right| \geq \delta \sqrt{n} \right) \leq \sqrt{n} e^{-C\sqrt{n}} \leq e^{-C' \sqrt{n}}\,.$$
	Let $A$ denote the event in the above probability; then conditioned on $A^c$ and $t \in [t_1,t_2]$ we have $$n^{-1/2}\left|\left\{ a \in \lambda : a \geq f_\infty(t-\delta)/\sqrt{n}     \right\} \right| \leq \int_{f_\infty(t - \delta)}^\infty \frac{1}{\exp\left(\sum_{j \in J} \bbeta_j x^j\right) - 1}\,dx + \delta \leq t $$
	and so $f_{\lambda,n}(t) \leq f_{\infty}(t - \delta)$.  A symmetric argument gives $f_{\lambda,n}(t) \geq f_{\infty}(t + \delta)$.  Since $f_\infty$ is continuous and $[t_1,t_2]$ is compact, the function $f_\infty$ is in fact uniformly continuous $[t_1,t_2]$ and so we may choose $\delta$ small enough with respect to $\eps$ so that $$|f_\infty(t) - f_\infty(t + \delta)| \leq \eps$$ for all $t \in [t_1 - \delta, t]$.   This completes the proof of the Corollary.
	\fi
\end{proof}

\begin{proof}[Proof of Theorem \ref{thm:shape}]
	Recall that in order to show Theorem \ref{thm:shape}, it is sufficient to show \eqref{eq:limit-need}; let $A$ denote the event in \eqref{eq:limit-need} and note $$\P_{\lambda\sim\mathcal{P}(\bN)}(A) = \P_{\lambda \sim \mu_n}(A \,|\, \lambda \in \mathcal{P}(\bN ))\leq \mu_n(\mathcal{P}(\bN))^{-1} \P_{\lambda \sim \mu_n}(A) = O(n^{C} e^{-c \sqrt{n}}) = o(1)\,. $$
\end{proof}

\section*{Acknowledgements}  
We thank Dan Romik and Robin Pemantle for helpful comments on a draft of this paper.    WP supported in part by NSF grants DMS-1847451 and CCF-1934915.

	\bibliography{partitions}
	\bibliographystyle{abbrv}

	\appendix

	\section{Proof of Lemma~\ref{lem:EM-asymptotic}}
	\label{secEulerMac}

	We make use of the  Euler-Maclaurin summation formula which allows us to compare a sum to a corresponding integral with an explicit error term.
	\begin{lemma}(Euler-Maclaurin Formula). \label{lem:EM}  Let $f: [a,b] \to \mathbb{C}$ be a smooth function.  Then \begin{align*}
		\sum_{k = a}^b f(k) = \int_a^b f(x)\,dx + \frac{f(a) + f(b)}{2} + \frac{f'(b) - f'(a)}{12} - \int_a^b f''(x)\frac{P_2(x)}{2}\,dx
		\end{align*}
		where $P_2$ is the second periodized Bernoulli polynomial $$P_2(x) = (x - \lfloor x \rfloor)^2 - (x - \lfloor x \rfloor) + \frac{1}{6}\,.$$
	\end{lemma}
	
	The following easy consequence of Stirling's formula is closely related to the Wallis product.
	\begin{fact}
		$$\int_1^\infty \frac{P_2(x)}{2x^2}\,dx = \frac{1}{2}\log(2\pi) - \frac{11}{12}\,.$$
	\end{fact}
	\begin{proof}
		Stirling's formula states $$\sum_{k = 1}^N \log k = N (\log N - 1) + \frac{1}{2}\log N + \frac{1}{2} \log(2\pi) + o(1)\,.$$
		Applying Lemma \ref{lem:EM} to the left-hand side and taking $N \to \infty$ completes the proof.
	\end{proof}

	To prove Lemma~\ref{lem:EM-asymptotic}, first note 
	\begin{equation} \label{eq:entropy-pieces}
	G\left (\frac{1}{e^a -1}\right) =  \frac{a}{e^a - 1}  -\log(1 - e^{-a})   \,.
	\end{equation}
	We will set $a = \sum_{j \in J} \bsg_j (tk)^j$ and will find asymptotics for each of these two terms.   To simplify notation in this section we will set 
	\begin{equation}
	p(x) = p_{\bsg}(x) = \sum_{j \in J} \bsg_j x^j \,.
	\end{equation} 
	
	\begin{lemma}\label{lem:EM-piece-2} As $t \to 0^+$ we have 
		\begin{align*}
\sum_{k \geq 1} \frac{p(tk)}{e^{p(tk)}  -1} = t^{-1}\int_0^\infty \frac{p(x)}{e^{p(x)}-1 }\,dx  -\frac{1}{2}\left(\one_{\jmin \geq 1} + \frac{\bsg_0}{e^{\bsg_0} - 1}\one_{\jmin = 0} \right) + o(1)\,.
		\end{align*}

	\end{lemma}
	\begin{proof}
		Set \begin{align*} f_{j}(t,x) :=
		\frac{d^j}{dx^j} \frac{p(tx)}{e^{p(tx)} - 1}
		\end{align*}
		for $j \in \{1,2\}$ and $$E:= -\frac{f_1(t,1)}{12} - \frac{1}{2}\int_1^\infty f_2(t,x)P_2(x)\,dx\,.$$
		
		By the Euler-Maclaurin formula (Lemma \ref{lem:EM}) we have \begin{align*}
		\sum_{k \geq 1} \frac{p(tk)}{e^{p(tk)} - 1 } &= \int_1^\infty \frac{p(tx)}{e^{p(tx)} - 1 }\,dx + \frac{1}{2}\frac{p(t)}{e^{p(t)} - 1}+ E \\
		&= t^{-1}\int_0^\infty \frac{p(x)}{e^{p(x)} - 1}\,dx - \frac{1}{2}\left(\one_{\jmin \geq 1} + \frac{\bsg_0}{e^{\bsg_0} - 1}\one_{\jmin = 0} \right)  +o(1)+E\,.
		\end{align*}
		To show $E = o(1)$ note first that $f_1(t,1) = O(t)$.  Define $g(s) := \frac{s}{e^s - 1}$ and compute directly that $g'$ and $g''$ decay exponentially as $s\to\infty$ and are both uniformly bounded.  We then see that $$f_2(t,x) = t^2\left( g''(p(tx)) (p'(tx))^2 + g'(p(tx))p''(tx)  \right) \,.$$
		Changing variables by setting $s = tx$ we see \begin{align*}\left|\int_1^\infty f_{2}(t,x)P_2(x) \,dx\right| &\leq t \int_0^\infty | g''(p(s))(p'(s))^2 + g'(p(s))p''(s)| \,ds\\
		&= O(t) \,,
		\end{align*}
		since $g''$ and $g'$ are uniformly bounded and decay exponentially.
	\end{proof}
	
	Before dealing with the other term in \eqref{eq:entropy-pieces}, we need two lemmas.
	
	\begin{lemma}\label{lem:EM-limit1} 
		We have
		$$ \lim_{t \to 0} \int_{1}^\infty \frac{ t^2 p''(t x)    }{e^{p(t x)} - 1 } P_2(x)\,dx = \jmin(\jmin-1) \int_1^\infty \frac{P_2(x)}{x^2}\,dx \,.$$
	\end{lemma}
	\begin{proof}
		We first claim that there is a constant $M$ so that for all $x > 0$ we have 
		$$\frac{ t^2 p''(t x)    }{e^{p(t x)} - 1 }\leq \frac{M}{x^2}\,.$$
		Multiplying both sides by $x^2$ and setting $b = t x$ we need to show that the function
		 $$ \frac{b^2  p''(b) }{e^{p(b)} - 1}$$
		  is uniformly bounded for all $b > 0$.  For $b$ near zero, the denominator is $\Omega( b^{\jmin})$ and the numerator is $O(b^2  \cdot b^{\jmin-2}) = O(b^{\jmin})$.  This shows that the function is bounded for $b$ in a neighborhood of $0$.  Conversely, the function tends to zero as $b$ tends to infinity, thus showing the desired inequality.  
		
		To prove the lemma, we apply the dominated convergence theorem along with the fact that for each fixed $x$ we have the desired limit.
	\end{proof}
	
	By a similar argument, we see the following.
	\begin{lemma}\label{lem:EM-limit2}
		We have
		 $$ \lim_{t \to \infty} \int_{1}^\infty \frac{e^{p(t x)} t^2 (p'(t x))^2    }{\left(e^{p(t x)} - 1 \right)^2} P_2(x)\,dx = \jmin^2 \int_1^\infty \frac{P_2(x)}{x^2}\,dx \,.$$
	\end{lemma}

	\begin{lemma} \label{lem:EM-piece-1} As $t \to 0^+$ we have 
		\begin{align*}
		\sum_{k \geq 1} \log(1 - \exp(-p(tk))) &=  t^{-1}\int_0^\infty \log(1 - \exp(-p(y)))\,dy - \frac{\one_{\jmin = 0}}{2}\log(1 - e^{-\bsg_0}) \\
		&\quad + \one_{\jmin \geq 1}\left(\frac{ \jmin}{2}\log(2\pi) -\frac{1}{2} \log(\bsg_{\jmin}) + \frac{\jmin}{2}\log(1/t)\right) + o(1)\,.
		\end{align*}
	\end{lemma}
	\begin{proof}
		Note that \begin{align*}
		\frac{d}{dx} \log \left(1 - e^{-p(tx)}\right) &=  \frac{t p'(tx)}{e^{p(tx)} - 1}  =: f_1(t,x)\\
		\frac{d^2}{dx^2} \log \left(1 - e^{-p(tx)}\right) &= \frac{ t^2 p''(tx)}{e^{p(tx)} - 1 }   - \frac{e^{p(tx)}t^2 (p'(tx))^2 }{(e^{p(tx)} - 1)^2} \\
		&=: f_2(t,x) \,.
		\end{align*}
		By the Euler-Maclaurin formula, we then have \begin{align*}
		\sum_{k \geq 1} \log(1 - \exp(-p(tk))) &=  \int_1^\infty \log(1 - \exp(-p(tx)))\,dx + \frac{1}{2} \log \left(1 - e^{-p(t)} \right) \\
		&- \frac{1}{12} f_1(t,1) - \frac{1}{2}\int_1^\infty f_2(t,x) P_2(x)\,dx\,.
		\end{align*}
		
		Each piece will be dealt with separately.  Write \begin{align*}
		\int_1^\infty \log(1 - \exp(-p(tx)))\,dx &=t^{-1} \int_0^\infty \log( 1 - \exp(-p(y)))\,dy - \int_0^1 \log(1 - \exp(-p(tx)))\,dx\,. 
		\end{align*}
		Note if $\jmin = 0$ then $$\int_0^1 \log(1 - \exp(-p(tx)))\,dx = \log(1 - e^{-\bsg_0}) + o(1)$$
		and if $\jmin \geq 1$ then 
		\begin{align*}
		\int_0^1 \log(1 - \exp(-p(tx)))\,dx &= \int_0^1 \log( p(tx))\,dx + o(1) \\
		&= \int_0^1 \log(\bsg_{\jmin} t^{\jmin} x^{\jmin})\,dx + o(1) \\
		&= \log(\bsg_{\jmin}) + \jmin \log t - \jmin +o(1)\,.
		\end{align*}
		
		Now compute 
		$$\frac{1}{2} \log \left(1 - e^{-p(t)} \right) = \frac{\one_{\jmin = 0}}{2}\log\left(1 - e^{-\gamma_0} \right) + \frac{\one_{\jmin \geq 1}}{2}\log(\bsg_{\jmin} t^{\jmin}) + o(1)$$
		and $$-\frac{1}{12} f_1(t,1) \sim -\frac{1}{12}\cdot \frac{\jmin t \gamma_{\jmin} t^{\jmin-1}}{\gamma_{\jmin} t^{\jmin}} =-\frac{\jmin}{12}\,.$$
		Applying the previous two lemmas to deal with the $P_2$ term finishes the proof.
	\end{proof}
	
	\begin{proof}[Proof of Lemma \ref{lem:EM-asymptotic}]
		
		Combining \eqref{eq:entropy-pieces} with Lemmas \ref{lem:EM-piece-2} and \ref{lem:EM-piece-1} completes the proof.
	\end{proof}	
	
\end{document}